\documentclass[10pt]{article}
\usepackage{amsmath}
\usepackage{amsfonts}
\usepackage{amsthm}

\newcommand\eq[1] {(\ref{#1})}

\newtheorem{remark}{Remark}[section]
\newtheorem{theorem}{Theorem}[section]
\newcommand{\x}{\mathbf{x}}
\newcommand{\ga}{\mathbf{g}^{(\alpha)}}
\newcommand{\E}{\mathbf{E}}

\newcommand{\J}{\mathbf{J}}

\newcommand{\ha}{\mathbf{h}^{(\alpha)}}

\newcommand{\ca}{c^{(\alpha)}}

\usepackage{graphicx}

\newcommand{\bfm}[1]{\mbox{\boldmath ${#1}$}}
\newcommand{\nonum}{\nonumber \\}
\newcommand{\beqa}{\begin{eqnarray}}
\newcommand{\eeqa}[1]{\label{#1}\end{eqnarray}}
\newcommand{\beq}{\begin{equation}}
\newcommand{\eeq}[1]{\label{#1}\end{equation}}
\newcommand{\Grad}{\nabla}
\newcommand{\Div}{\nabla \cdot}
\newcommand{\Curl}{\nabla \times}

\newcommand{\Real}{\mathop{\rm Re}\nolimits}
\newcommand{\Imag}{\mathop{\rm Im}\nolimits}
\newcommand{\Tr}{\mathop{\rm Tr}\nolimits}

\newcommand{\lang}{\langle}
\newcommand{\rang}{\rangle}
\newcommand{\Md}{\partial}

\newcommand{\Ga}{\alpha}
\newcommand{\Gb}{\beta}
\newcommand{\Gd}{\delta}
\newcommand{\Ge}{\epsilon}

\newcommand{\Gg}{\gamma}
\newcommand{\Gc}{\chi}

\newcommand{\Gk}{\kappa}

\newcommand{\Gl}{\lambda}

\newcommand{\Gm}{\mu}

\newcommand{\Gs}{\sigma}

\newcommand{\Gj}{\tau}

\newcommand{\Go}{\omega}

\newcommand{\GO}{\Omega}

\newcommand{\BGe}{\bfm\epsilon}
\newcommand{\BGve}{\bfm\varepsilon}

\newcommand{\BGj}{\bfm\tau}

\def\Be{{\bf e}}

\def\Bj{{\bf j}}

\def\Bn{{\bf n}}

\def\Bu{{\bf u}}

\def\Bx{{\bf x}}

\def\BC{{\bf C}}

\def\BE{{\bf E}}

\def\BI{{\bf I}}
\def\BJ{{\bf J}}

\def\BR{{\bf R}}
\def\BS{{\bf S}}

\def \ba {\begin{array}}
\def \ea {\end{array}}
\newcommand{\xione}{\xi^{(1)}}
\newcommand{\xitwo}{\xi^{(2)}}
\newcommand{\etaone}{\eta^{(1)}}
\newcommand{\etatwo}{\eta^{(2)}}
\newcommand{\psione}{\psi^{(1)}}
\newcommand{\psitwo}{\psi^{(2)}}

\newcommand{\Eoo}{\BE_1^{(1)}}
\newcommand{\Eot}{\BE_1^{(2)}}
\newcommand{\Eto}{\BE_2^{(1)}}
\newcommand{\Ett}{\BE_2^{(2)}}

\newcommand{\n}{\mathbf{n}}

\newcommand{\aE}{\langle\E\rangle}
\newcommand{\aJ}{\langle\J\rangle}
 
\newcommand{\lla}{\left\langle}
\newcommand{\rra}{\right\rangle}
\usepackage{graphics}
\usepackage{graphicx}
\usepackage{amssymb}
\usepackage{wrapfig}
\graphicspath{ {Images/} }
\usepackage{caption,subcaption}
\begin{document}
\vspace{-1in}
\title{Criteria for guaranteed breakdown in two-phase inhomogeneous bodies}
\author{Patrick Bardsley, Michael S. Primrose, Michael Zhao, Jonathan Boyle, \\ Nathan Briggs, Zoe Koch,
and Graeme W. Milton\\
\small{Department of Mathematics, University of Utah, Salt Lake City UT 84112, USA}}
\date{}
\maketitle
\begin{abstract}
Lower bounds are obtained on the maximum field strength in one or both phases in a body containing two-phases. 
These bounds only incorporate boundary data that can be obtained from measurements at the surface of the body, 
and thus may be useful for determining if breakdown has necessarily occurred in one of the phases, or that some other
nonlinearities have occurred. It is assumed the response of the phases is linear up to the point of electric, dielectric,
or elastic breakdown, or up to the point of the onset of nonlinearities.
These bounds are calculated for conductivity, with one or two sets of boundary conditions,
for complex conductivity (as appropriate at fixed frequency when the wavelength is much larger than the body, i.e., 
for quasistatics), and for two-dimensional elasticity. Sometimes the bounds are optimal when the field
is constant in one of the phases, and using the algorithm of Kang, Kim, and Milton (2012) a wide variety of inclusion shapes
having this property, for appropriately chosen bodies and appropriate boundary conditions, are numerically constructed.
Such inclusions are known as $E_\Omega$-inclusions. 
\end{abstract}
\vskip2mm

\vskip 5mm

\section{Introduction}
\setcounter{equation}{0}
An inverse problem of obviously major practical significance is the detection of cracks inside a body using measurements at the boundary of the body. If the body in the absence
of the crack is a homogeneous material, such as a metal, the calculation of the fields inside the body is a straightforward numerical problem and in this way
cracks can be detected. But, for example, with the advent of aeroplanes built from carbon fibre composite materials it is becoming increasingly important to detect
cracks in composites, or more generally in inhomogeneous bodies. Ideally one would like to solve the inverse problem of locating the position of a crack in an inhomogeneous 
body with an unknown configuration of the phases in the body,
but at the very least one would like to be able to identify those boundary fields that necessarily imply there is a crack in the body, or that some other
breakdown in the equations has occured inside the body.  It is the purpose of this paper to identify such boundary fields. 
While many of the arguments are elementary and 
while it seems highly likely that the results presented here can be improved, the paper is perhaps the first to embark on this significant problem and has the goal
of introducing the inverse problem community to it, so that further progress can be made. Another important, but related, detection problem is in 
breast cancer, where the breast is again an inhomogeneous body, perhaps modelled a two phase medium, where the phases are the glandular tissue (containing the milk-producing 
cells) and adipose tissue (fatty cells). In this context large interior fields, or a breakdown in the two-phase equations, could signal breast cancer.

A material often breaks down if the local field exceeds a certain critical value. This may be the current field strength which causes melting in a conducting material, 
the electric field strength which causes dielectric breakdown in an insulating material, or the value of the stress field which causes plastic yielding or cracking
in an elastic material. Usually one wants to 
avoid this and so it is helpful to have some idea of the maximum field within a body $\GO$ from measurements of the
(voltage, current flux) or (displacement, traction) at the boundary $\Md\GO$ of the body. If the body is homogeneous then we may
numerically solve for the fields in the interior and thus calculate explicitly the maximum field. However if the body is 
inhomogeneous, say containing two phases in an unknown geometry as illustrated in figure 1, then we cannot do this but still we would like to say 
something rigorous about the field inside. As we are making no assumptions about the geometry there could be sharp corners
or other singularities in the surface between phases inside the body, and these will lead to infinite local fields in the absence of breakdown or nonlinearities. Thus 
all we can hope for are lower bounds on the magnitude of the maximum local field, where the maximum is taken over 
one or both phases. Thus we want to identify boundary data which are certainly dangerous in the sense that they necessarily imply that breakdown has occurred inside the body, or that some other nonlinearities must have occurred. We assume that the response of each phase is linear up to the point of
breakdown, or up to the point of onset of nonlinearities.  

\begin{figure}
	\centering
	\includegraphics[width=0.50\textwidth]{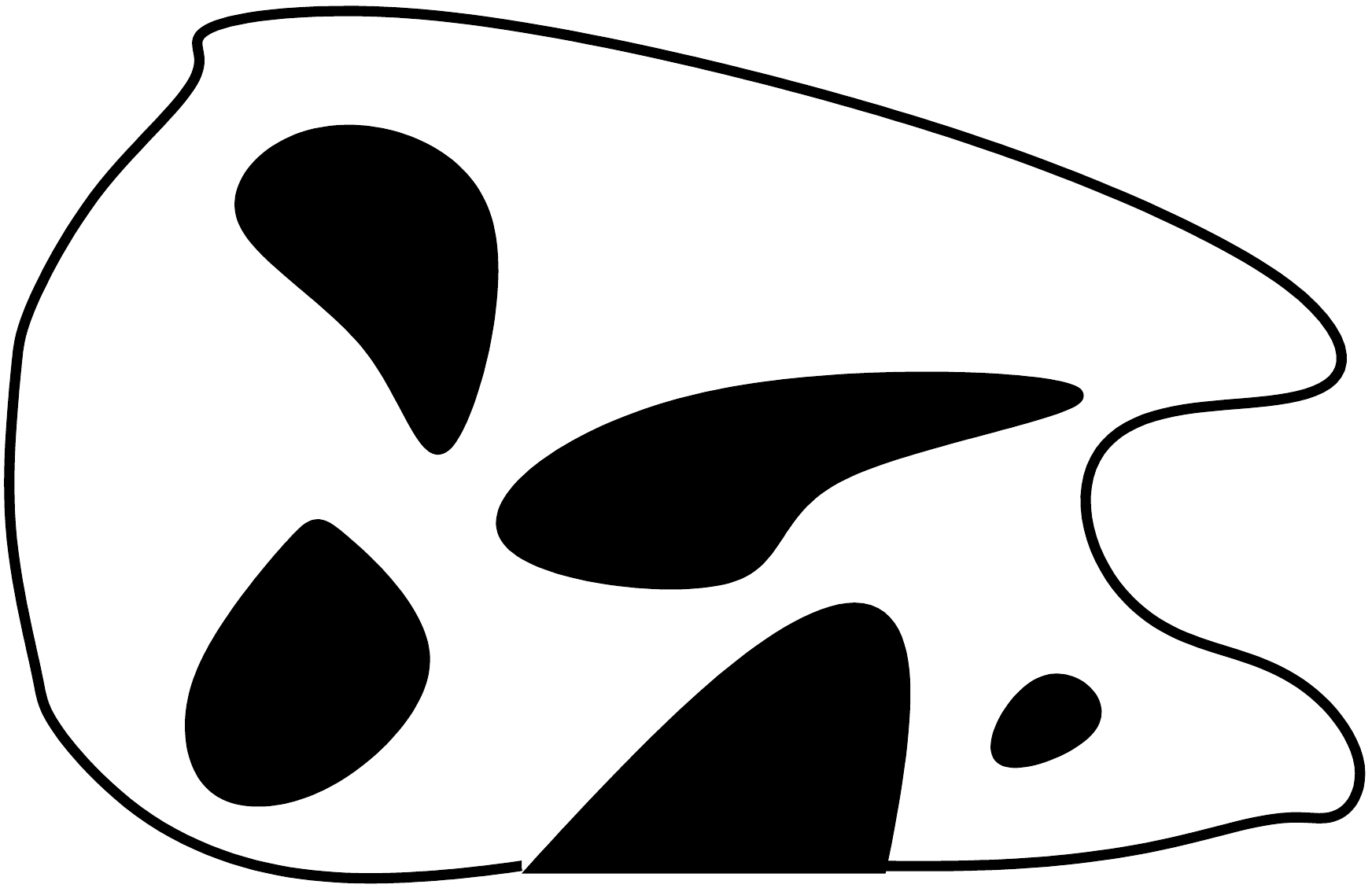}
	\caption{The body $\GO$ containing phase 1, in black, occupying the region $\GO_1$ and phase 2, in white, 
        occupying the region $\GO_2$.}
	\label{fig:Omega}
\end{figure}

Similar questions have been addressed before in the context of periodic or statistically homogeneous composite materials of infinite extent: rigorous
bounds have been obtained on the effective yield surface of polycrystalline materials \cite{Bishop:1951:TPD, Kohn:1999:SMP, Nesi:2000:IBY, Goldsztein:2001:RPP, Garroni:2003:STD}, on the set of recoverable strains of polycrystalline shape memory materials \cite{Bhattacharya:1997:EEM}, and on the lowest value of the maximum field magnitude (or maximum of some norm of the field, for matrix valued fields) within two phase linear composites 
\cite{Lipton:2004:OLB, Lipton:2005:OLB, Lipton:2006:OLB, He:2007:LBS, Alali:2009:OLB, Chen:2010:OLB, He:2010:LSS, Alali:2012:NBL, Liu:2014:GIM}. Also results have 
been obtained on the lowest value of the maximum of some norm of the field for one or more inclusions in an infinite body when uniform fields are imposed at infinity \cite{Wheeler:2004:IMS, Liu:2014:GIM}. To our knowledge such bounds have not been obtained for a two-phase body $\GO$ of finite extent with general boundary conditions at $\Md\GO$ and it is the purpose of this paper to address this. 

While the bounds we obtain are very crude (because we bound the average in each phase of the square of fields, by the square of the maximum field) we believe 
they are the first rigorous inequalities addressing this problem, and as such should serve as a benchmark for future progress. Furthermore, they  are sharp for certain geometries. In particular, many of the bounds are sharp when the field in one phase is constant. Numerous examples have been found of periodic or statistically homogeneous two-phase composites having the property that the field is constant in one phase
\cite{Maxwell:1954:TEM, Hashin:1962:EMH, Milton:1980:BCD, Milton:1981:BCP, Tartar:1985:EFC, Lurie:1986:EEC, Norris:1985:DSE, Milton:1986:MPC, Francfort:1986:HOB, Grabovsky:1995:MMEa, Vigdergauz:1986:EEP, Vigdergauz:1994:TDG, Grabovsky:1995:MMEb, Vigdergauz:1996:RLE, Vigdergauz:1999:EMI, Vigdergauz:1999:CES, Sigmund:2000:NCE, Benveniste:2003:NER, Liu:2007:PIM}. Also sets of inclusions in an infinite matrix have been found such that the field in them is uniform when a uniform field is applied at infinity
\cite{Cherepanov:1974:IPP, Kang:2008:IPS, Liu:2008:SEC, Liu:2014:GIM, Dai:2015:USF}. Liu, James and Leo \cite{Liu:2007:PIM} call these inclusions $E$-inclusions.
For a single inclusion in a matrix with a uniform field at infinity the field is uniform when the inclusion is an ellipsoid 
\cite{Poisson:1826:SMS, Maxwell:1954:TEMb, Eshelby:1957:DEF, Eshelby:1961:EII, Khachaturyan:1966:SQC, Willis:1981:VRM} and it was conjectured by Eshelby 
\cite{Eshelby:1957:DEF, Eshelby:1961:EII} that this is the only simply connected inclusion with this property. Eshelby's conjecture was proved for planar elasticity
in \cite{Sendeckyj:1970:EIP}, for two-dimensional conductivity or equivalently antiplane elasticity in \cite{Ru:1996:EIA}, for three-dimensional conductivity
and elasticity when the uniformity property holds for all uniform applied fields in \cite{Kang:2008:SPS,Liu:2008:SEC}, and in three-dimensional elasticity 
when it holds for two independent uniform applied fields in \cite{Ammari:2010:PSE}. On the other hand Liu \cite{Liu:2008:SEC} has shown that for three-dimensional conductivity
with a single uniform applied field there are nonellipsoidal inclusions which have a uniform field inside.

Inclusions with a uniform field inside retain this property if we truncate the material to a body $\GO$ of finite extent and apply appropriate boundary conditions.
However given a body $\GO$ there could exist a wider class of inclusions called $E_\GO$-inclusions contained within $\GO$ for which the field is uniform for some
boundary condition: $E$-inclusions lying inside $\GO$ are $E_\GO$-inclusions, but the converse is not true. For two-dimensional conductivity simply connected 
$E_\GO$-inclusions were constructed by Kang, Kim and Milton \cite{Kang:2011:SBV}. Here we show that these inclusions remain $E_\GO$-inclusions under appropriate affine transformations, and that they are also $E_\GO$-inclusions for elasticity with appropriate boundary conditions.

\section{Real conductivity with one boundary condition}
\setcounter{equation}{0}
In this section, we consider the equations of real conductivity in the body $\GO$ in the absence of source terms:
\beq \BJ(\Bx)=\Gs(\Bx)\BE(\Bx),\quad \Div\BJ=0,\quad \BE=-\Grad V, \eeq{0.1}
in which $\BJ$ is the current field, $\BE$ is the electric field, $V$ is the potential, and 
\beq \Gs(\Bx)=\Gs^{(1)}\Gc_1(\Bx)+\Gs^{(2)}\Gc_2(\Bx) \eeq{0.1a}
is the (scalar valued) local conductivity, where $\Gs^{(1)}$ and $\Gs^{(2)}$ are the scalar conductivities of phases 1 and 2, respectively,
and $\Gc_i(\Bx)$ is the characteristic function of phase $i$ taking the value $1$ in phase i, and zero outside it 
(thus $\Gc_1(\Bx)+\Gc_2(\Bx)=1$ within the body). Breakdown at a given point
$\Bx$ in phase $\Ga=1,2$  is assumed to depend only the local electric field $\BE(\Bx)$ at that point. As the phases are isotropic 
it should only depend on the magnitude $|\BE(\Bx)|$. Thus the local criterion for breakdown in phase $\Ga$ at point $\Bx$ is that
\beq |\BE(\Bx)|\geq c^{(\Ga)}, \eeq{0.1b}
and conversely if $|\BE(\Bx)|< c^{(\Ga)}$ we will say the material has not broken at the point $\Bx$. From boundary measurements we can
determine the potential $V$ and the current flux $\BJ\cdot\Bn$ at the boundary $\Md\GO$. We seek criteria which enable us to say with 
certainty that the boundary measurements imply  breakdown has occurred somewhere inside the body (assuming the linear equations \eq{0.1} hold
up to the point of breakdown). 

\subsection{In which phase does breakdown first occur?}
In the two-dimensional situation, we have the following result.

\begin{theorem}
	In two-dimensions $|\BE|$ takes its maximum value over a given phase on the boundary of that phase, which may be at the interface between phases, or on the boundary of the
body $\GO$.
\end{theorem}
\begin{proof}
	In two dimensions, the conductivity equations take the form
\beq    \E = \left[
	\begin{array}{ cc }
	-\dfrac{\Md V}{\Md x}, & -\dfrac{\Md V}{\Md y} 
	\end{array} \right], \quad \BJ=\Gs\BE,\quad \nabla \cdot \textbf{J} = \frac{\Md\J_1}{\Md x}+\frac{\Md\J_2}{\Md y}=0.
\eeq{1.1}
	Let us set $z = x + iy$. Now $V$ is harmonic in each connected part of one phase and there is no net charge inside the connected
part if it is multiply connected. So within this connected part $V$
is the real part of some analytic function $g = V + iW$, and we have
\beq \E = \Real \left[
	\begin{array}{ cc }
	-\dfrac{\Md g}{\Md x}, & -\dfrac{\Md g}{\Md y} 
	\end{array} \right] = \Real \left[
	\begin{array}{ cc }
	-\dfrac{\Md g}{\Md z}\dfrac{\Md z}{\Md x}, & -\dfrac{\Md g}{\Md z}\dfrac{\Md z}{\Md y}
	\end{array} \right] = \Real \left[
	\begin{array}{ cc }
	-\dfrac{\Md g}{\Md z}, & -i\dfrac{\Md g}{\Md z}
	\end{array} \right].
\eeq{1.2}
	Then if we let $f(z) = \Md g/\Md z$, $\textbf{E} = (-\Real f, \Imag f)$. But $f$ is analytic, so by the maximum modulus principle $|f|$ takes its maximum on the boundary of this connected part of the phase. On the other hand, $|\BE| = |f|$, so the maximum value of $\BE$ occurs on the boundary between the two phases, or at the boundary of $\GO$
\end{proof}
With this, it follows that breakdown must occur at the boundary $\Md\GO$ or at the interface between the phases. The following theorem gives conditions under which we can know whether the breakdown occurs first in phase 1 or first in phase 2. 
\begin{theorem}
	In two-dimensions, electrical breakdown definitely occurs in phase 1 first, if it does not first occur at the boundary $\Md\GO$, if
\beq  (c^{(2)})^2>(c^{(1)})^2\max\{(\sigma^{(1)}/\sigma^{(2)})^2, 1\}, \eeq{1.10}
and definitely occurs in phase 2 first, if it does not first occur at the boundary $\Md\GO$, if
\beq  (c^{(2)})^2<(c^{(1)})^2\min\{(\sigma^{(1)}/\sigma^{(2)})^2, 1\}. \eeq{1.11}
\end{theorem}
\begin{proof}
Let $\E^{(1)}$ denote the field in phase 1 and $\BE^{(2)}$ denote the field in phase 2. At a point on the boundary between the phases, assuming the boundary is smooth at that point, we have the decomposition
\beq
\E^{(1)}=\E^{(1)}_n+\E^{(1)}_t, \quad
\E^{(2)}=\E^{(2)}_n+\E^{(2)}_t,
\eeq{1.11a}
where $n$ and $t$ label the normal and tangential components of each electric field, and these field components satisfy the jump conditions
\beq \Gs^{(2)}\E^{(2)}_n=\Gs^{(1)}\E^{(1)}_n, \quad \E^{(2)}_t=\E^{(1)}_t, \eeq{1.11b}
implied by continuity of the flux, and continuity to the potential at the interface.
To motivate the conditions which appear in \eq{1.10} and \eq{1.11} let us suppose that breakdown simultaneously begins to occur in both phase 1 and phase 2 at one point on
the interface between the phases. At that point we have
\beq
\frac{|\E^{(1)}|^2}{(c^{(1)})^2}=\frac{|\E^{(2)}|^2}{(c^{(2)})^2}=1
\Rightarrow (c^{(2)})^2 |\E^{(1)}|^2= (c^{(1)})^2 |\E^{(2)}|^2,
\eeq{1.11c}
and hence 
\beq
(c^{(2)})^2 (|\E^{(1)}_n|^2+|\E^{(1)}_t|^2)
=(c^{(1)})^2 (|\E^{(2)}_n|^2+|\E^{(2)}_t|^2).
\eeq{1.11d}
Substituting the jump conditions \eq{1.11b} into this and gathering terms, we see that
\beq
[(c^{(2)})^2-(c^{(1)})^2(\sigma^{(1)}/\sigma^{(2)})^2]|\E^{(1)}_n|^2=[(c^{(1)})^2-(c^{(2)})^2]|\E^{(1)}_t|^2.
\eeq{1.20}
If $(c^{(2)})^2-(c^{(1)})^2(\sigma^{(1)}/\sigma^{(2)})^2$ and $(c^{(1)})^2-(c^{(2)})^2$ have opposite signs then this equation will have no real solution 
for $|\E^{(1)}_n|$ and $|\E^{(1)}_t|$, other than the trivial solution $|\E^{(1)}_n|=|\E^{(1)}_t|=0$ which will not correspond to breakdown.
If in particular \eq{1.10} holds then the $=$ sign in \eq{1.20} can be replaced by a $>$ sign for all
nonzero $\E^{(1)}$ and tracing back the equations one concludes
that $(c^{(2)})^2 |\E^{(1)}|^2> (c^{(1)})^2 |\E^{(2)}|^2$, implying that any point on the interface between the phases, breakdown will first occur in 
phase 1. Similarly, if \eq{1.11} holds then the $=$ sign in \eq{1.20} can be replaced by a $<$ sign
for all nonzero $\E^{(1)}$, and tracing back the equations one concludes
that $(c^{(2)})^2 |\E^{(1)}|^2<(c^{(1)})^2 |\E^{(2)}|^2$, implying that any point on the interface between the phases, breakdown will first occur in 
phase 2. We remark if neither \eq{1.10} nor \eq{1.11} holds then \eq{1.20} may have a nontrivial solution for  $|\E^{(1)}_n|^2$ and $|\E^{(1)}_t|^2$,
and it seems likely that we cannot determine in which phase breakdown first occurs in without knowing the phase geometry and boundary fields. 
\end{proof}
\subsection{Elementary Breakdown Criteria}
If the materials in the body have not broken down anywhere, then they certainly will not have broken down at the surface $\Md\GO$. 
If we know $V(\x)$ at $\Md \Omega$, then we know $\mathbf{t}\cdot\Grad V$ for any unit vector $\mathbf{t}$ tangential to $\Md\Omega$. Also because
the materials are isotropic $\J\cdot\n/\sigma$ gives us the normal component of $\E$. Assuming we know $\sigma$ at the boundary, we can determine $\E$ at the boundary. \\

\noindent \textbf{Criterion 1}. In two dimensions if the material has not broken down in phase $\Ga=1,2$, then the inequality
\beq
	c^{(\alpha)} > |\E|= \sqrt{\left(\frac{\J\cdot\n}{ \sigma^{(\Ga)}}\right)^2+(\Grad V \cdot \mathbf{t})^2}
\eeq{1.30}
must be satisfied for all points at the surface $\Md\GO$ which are in phase $\Ga$. 
In three dimensions, if $\n, \textbf{t}_1, \textbf{t}_2$ are three orthonormal normal and tangential vectors, then if the material has not broken down in phase $\Ga=1,2$, the inequality
\beq
	c^{(\Ga)} > |\E |= \sqrt{\left(\frac{\J\cdot\n}{ \sigma^{(\Ga)}}\right)^2+(\Grad V \cdot \textbf{t}_1)^2 + (\Grad V\cdot \textbf{t}_2)^2}
\eeq{1.31}
must be satisfied for all points at the surface $\Md\GO$ which are in phase $\Ga$.

\subsubsection{Breakdown criteria based on the average fields}
Here we find expressions for the average over each phase of the electric field, and using the fact that the variance, over each phase,
of the electric field must be nonnegative, we obtain simple breakdown criteria.

Using the fact that $\textbf{E} = - \nabla V$, we have 
\beq \aE
= \frac{1}{|\Omega |} \int_{\Omega} -\Grad V~d\textbf{x}  
= \frac{1}{|\Omega|}\int_{\Md\Omega}-V\Bn~dS,
\eeq{1.32}
where $\Bn$ is the outward normal to the boundary $\Md\GO$. 
Now consider $\langle \J \rangle$. Let $x_i$ be the $i$-th coordinate. Then since $\nabla \cdot \J = 0$, we find $\nabla \cdot (x_i \J) = J_i$ where $J_i$ is the $i$-th component of $J$. Then directly by the divergence theorem, 
\beq \int_{\Omega} J_i ~d \textbf{x}  = \int_{\Omega} \Div (x_i\J) ~d \textbf{x}  = \int_{\Md \Omega} x_i (\J\cdot \n)~dS .
\eeq{1.33}
 Hence we can find both $\langle \textbf{E} \rangle$ and $\langle \J \rangle$ from the boundary measurements and from these we can determine 
\beq \aE_1= \frac{\langle \chi_1 \E \rangle }{f_1}=\frac{1}{|\Omega_1 |}\int_{\Omega_1}\E~d \Bx,\quad\mbox{ where }~ \Omega_1=\chi_1\Omega,
\eeq{1.34}
and 
\beq \aE_2= \frac{\langle \chi_2  \E \rangle }{f_2}=\frac{1}{|\Omega_2 |}\int_{\Omega_2}\E~d \Bx\quad\mbox{ where }~\Omega_2=\chi_2\Omega,
\eeq{1.35}
which represent the average over each phase of the electric field, and $\GO_1$ and $\GO_2$ are the regions occupied by phases 1 and 2 respectively.
To see this, notice that
\beqa \aE & = & \langle \chi_1 \E \rangle +\langle \chi_2 \E \rangle 
= f_1 \aE_1+ f_2 \aE_2, \nonum
\aJ& = & \langle \chi_1 \J\rangle + \langle \chi_2 \J\rangle
=\sigma_1 \langle \chi_1 \E\rangle + \sigma_2 \langle \chi_2 \E\rangle
=f_1\sigma_1\aE _1 +f_2\sigma_2\aE _2, \nonum
&~&
\eeqa{1.36}
which when solved for $\aE_1$ and $\aE_2$ give
\beq
\aE_1=\frac{\aJ-\sigma_2\aE}{f_1(\sigma_1-\sigma_2)},\quad \aE_2=\frac{\aJ-\sigma_1\aE}{f_2(\sigma_2-\sigma_1)},
\eeq{1.37}
where we have assumed that $\Gs_1\ne\Gs_2$. Analogous formulae to \eq{1.37} are well known in the theory of composites: see for example, equation (6) in
\cite{Polder:1946:EPM}.
Now from the positivity of the variance of the electric field in phase 1, and if the material has not
broken down in phase 1, we have
\beq 0\leq \lang[\Gc_1(\BE-\aE_1)\cdot(\BE-\aE_1)]\rang=\lang\Gc_1|\BE|^2\rang-f_1|\aE_1|^2\leq f_1[(c^{(1)})^2-|\aE_1|^2], \eeq{1.38}
with equality if and only if the field is constant in phase 1, having magnitude $|\BE|=c^{(1)}$. Similarly if the material has not broken down in
phase 2, then $(c^{(2)})^2\geq |\aE_2|^2$.

This gives us the following criterion. \\

\noindent \textbf{Criterion 2}. \textit{Let $\aE_{\Ga}$ be the average over phase $\Ga$ of the electric field, given in terms of the boundary data 
through \eq{1.37}, \eq{1.32} and \eq{1.33}. If neither phase has broken down then the inequality $|\aE_\Ga|\leq c^{(\Ga)}$ must hold. (Note that
the derivation of the inequality assumes that both phases have not broken down, not just phase $\Ga$.) } \\

A similar criterion for two phase periodic or statistically homogeneous composite materials (with a similar derivation) 
follows directly from the results of 
\cite{Lipton:2004:OLB}.

Notice that we also have
\beqa
\langle\J \cdot \E \rangle 
 & = & \frac{1}{|\Omega |}\left( \int_{\Omega_1} \sigma_1 |\E |^2+\int_{\Omega_2} \sigma_2 |\E |^2 \right) d \textbf{x} \nonum
 & \leq & \frac{1}{|\Omega |}\left( \sigma^{(1)} \int_{\Omega_1}  (c^{(1)})^2 d \textbf{x} + \sigma^{(2)}\int_{\Omega_2} (c^{(2)})^2 \right) d \textbf{x}  \nonum
 & \leq & \sigma^{(1)}(c^{(1)})^2f^{(1)}+\sigma^{(2)}(c^{(2)})^2f^{(2)},
\eeqa{1.39}
and 
\beq \int_{\Omega} \J\cdot \E \ d \x = \int_{\Omega} -\J \cdot \Grad V d \x = -\int_{\Md\Omega} V (\J \cdot \n) dS,
\eeq{1.40}
where we have made use of the fact that $\Div(V\J)=(\Grad V)\cdot\J+V\Div\J = \Grad V \cdot \J$. This gives us the following criterion.\\

\noindent \textbf{Criterion 3}. \textit{If the material has not broken down, $\langle \J \cdot \E \rangle$ satisfies the bounds \eq{1.39} and is given in 
terms of the boundary data by \eq{1.40}.}\\

In contrast to Criterion 2, this may still be useful even if $\aE_{1}$ and $\aE_{2}$  are both zero, as may happen if the body, phase geometry, and 
boundary conditions have appropriate symmetries.

\subsection{Improved criteria by perturbing the conductivity}

Criteria 2 and 3 are derived from inequalities on $\lang\Gc_1|\BE|^2\rang$ and $\lang \Gc_2|\BE|^2\rang$. If we can determine these quantities
directly from suitable measurements this will lead to improved breakdown criteria. It may be the case that the conductivities $\Gs_1$ and $\Gs_2$ can be perturbed
by a small amount, by for example changing the temperature, or by introducing boundary conditions which oscillate with time at some low fixed frequency
$\Go$ (in which case the conductivities may have a small imaginary part). 

Suppose we fix the potential $V=V_0$ on $\Md\Omega$ and that under the perturbation a quantity $a$ goes to $a+\delta a$, except $\delta V=0$ on $\Md\Omega$, i.e. the surface maintains the same voltage.

Then to first order in the perturbation we have
\beq \J+\delta\J =(\sigma+\delta\sigma)(\E+\delta\E)
\approx\sigma\E+(\delta\sigma)\E+\sigma(\delta\E),
\eeq{1.42}
implying
\beqa
	\int_\Omega (\E+\delta\E)\cdot(\J+\delta\J)~d \Bx
	 & \approx & \int_\Omega (\E+\delta\E)\cdot[\sigma\E+(\delta\sigma)\E+\sigma\delta\E]~d \Bx \nonum
	 & \approx & \int_\Omega \E\cdot\sigma\E~d \Bx +\int_\Omega \E\cdot(\delta\sigma)\E~d \Bx+2\int_\Omega \E\cdot(\sigma\delta\E)~d \Bx. \nonum & ~ &
\eeqa{1.43}
The quantities
\beqa &~&\int_\Omega (\E+\delta\E)\cdot(\J+\delta\J)\ d\Bx = -\int_{\Md\Omega} V(\J+\delta\J)\cdot\n~dS \quad \mbox{since $\delta V=0$ on $\Md\Omega$}, 
\nonum
&~&\int_{\Omega} \E\cdot\sigma\E  \ d\x  =  \int_{\Omega} \E\cdot\J~d \Bx=  -\int_{\Md\Omega} V (\J \cdot \n)~dS, \nonum
&~& \int_\Omega \E\cdot(\sigma\delta\E)~d \Bx  =  \int_\Omega \J\cdot\delta\E~d \Bx=-\int_{\Md\Omega}\delta V(\J\cdot\n)~dS=0\quad \mbox{since $\delta V=0$ on $\Md\Omega$}, \nonum & ~ &
\eeqa{1.44}
can all be evaluated from boundary data from the perturbed and unperturbed problems. Therefore, using \eq{1.43},
\beq
\int\E\cdot\delta\sigma\E~d \Bx
 =  \delta\sigma_1\int_{\Omega_1}|\E|^2~d \Bx+\delta\sigma_2\int_{\Omega_2}|\E|^2~d \Bx
\eeq{1.45}
can be determined from boundary data (to first order in the perturbation), together with 
\beq \int_{\Omega}\E\cdot\sigma\E~d \Bx
=\sigma_1\int_{\Omega_1}|\E|^2~d \Bx+\sigma_2\int_{\Omega_2}|\E|^2~d \Bx,
\eeq{1.46}
and these may be solved for $\int_{\Omega_1}|\E|^2 \ d\x$ and $\int_{\Omega_2}|\E|^2 \ d\x$ provided 
\beq \det \left[\begin{array}{c c}
\delta\sigma_1 & \delta\sigma_2
\\ \sigma_1 & \sigma_2
\end{array}\right] \neq 0.
\eeq{1.47}
Thus we obtain \\

\noindent \textbf{Criterion 4}. \textit{If the material has not broken down, we have the bound
\beq \int_{\Omega_\Ga}|\E|^2 \ d\x\leq |\GO|f_\Ga(c^{(\Ga)})^2, \eeq{1.48}
for $\Ga=1,2$ where if \eq{1.47} is satisfied the quantity of the left can be determined from boundary data obtained by perturbing the conductivities.}

\section{Real conductivity with two separate boundary conditions via the splitting method}
\setcounter{equation}{0}
In this section we consider conditions which guarantee breakdown occurs within the body, for at least one of two potentials separately 
applied to the boundary of the body $\GO$.
The equations we consider are now
\beq \Div\BJ_i=0,\quad \BE_i=-\Grad V_i,\quad \BJ_i(\Bx)=\Gs(\Bx)\BE_i, \eeq{0.1aa}
where the scalar conductivity $\Gs(\Bx)$ is real, and $i=1$ labels the fields associated with one set of boundary conditions, while $i=2$ labels the fields 
associated with the other set of boundary conditions. If the material has not broken down for both boundary conditions, the inequality
\beq |\E_m^{(\alpha)}(\x)|\leq \ca \eeq{0.2}
must be satisfied for all $\Bx\in\GO$, for $m=1,2$ and for $\alpha=1,2$,  where $\E_m^{(\alpha)}(\x) = \chi_\alpha(\x)\E_m(\x)$. Formally the equations \eq{0.2} are equivalent
to the quasistatic equations
\beq  \Div\BJ=0,\quad \BE=-\Grad V,\quad \BJ(\Bx)=\Gs(\Bx)\BE, \eeq{0.3}
with a real conductivity $\Gs(\Bx)$ but complex fields
\beq \BJ=\BJ_1+i\BJ_2, \quad \BE=\BE_1+i\BE_2, \quad V=V_1+iV_2, \eeq{0.4}
where a subscript $1$ here denotes the real part, while a subscript $2$ here denotes the imaginary part. However, the breakdown conditions \eq{0.2} 
are not generally appropriate
for the quasistatic equations, as discussed later in section 4.1. Despite this, the connection with the complex conductivity equations allows us to use much of the
analysis of Thaler and Milton \cite{Thaler:2015:BVI}, who derived bounds on the volume fractions of the two phases, from boundary measurements
using the splitting method introduced by Milton and Nguyen \cite{Milton:2011:BVF}. (See also the paper of Kang, Lim, Lee, Li, and Milton \cite{Kang:2014:BSI}
which addresses the problem of bounding
the volume fraction from boundary measurements with complex conductivities using the translation method, extending earlier work of Kang, Kim, and Milton \cite{Kang:2011:SBV}
and Kang and Milton \cite{Kang:2013:BVF3d} that bounded the volume fractions
using the translation method when the conductivities were real.) 
By contrast, we will assume the volume fractions are known, but instead find boundary measurements which necessarily signal that \eq{0.2} is violated. 

Thaler and Milton \cite{Thaler:2015:BVI} consider the quantity, for $\x \in \Omega, \mathbf{c}^{(\alpha)} \in \mathbb{R}^2$, and $\alpha = 1,2$
\begin{equation}\label{g}
\ga(\x;\mathbf{c}^{(\alpha)}) := \displaystyle\sum_{m=1}^2 c^{(\alpha)}_m \left[\E_m^{(\alpha)}(\x) - \dfrac{\chi_\alpha(\x)}{f_{\Ga}}\langle\E_m^{(\alpha)}\rangle\right],
\end{equation}
with $\E_m^{(\alpha)}(\x) = \chi_\alpha(\x)\E_m(\x)$. From the nonnegativity of the variance $\langle \ga \cdot \ga \rangle \geq 0$, they deduce that the symmetric matrix 
\beq \BS^{(\alpha)} = \left [
\begin{matrix}
A_{11}^{(\alpha)} - \dfrac{D_{11}^{(\alpha)}}{f_{\Ga}} && A_{12}^{(\alpha)} - \dfrac{D_{12}^{(\alpha)}}{f_{\Ga}} \\
A_{21}^{(\alpha)} - \dfrac{D_{21}^{(\alpha)}}{f_{\Ga}} && A_{22}^{(\alpha)} - \dfrac{D_{22}^{(\alpha)}}{f_{\Ga}} 
\end{matrix}
\right ] \eeq{0.6}
 must be positive semi-definite, where $A_{mn}^{(\alpha)} = \langle \E_m^{(\alpha)} \cdot \E_n^{(\alpha)} \rangle$ and $D_{mn}^{(\alpha)} = \langle \E_m^{(\alpha)} \rangle \cdot \langle \E_n^{(\alpha)} \rangle$. The average fields $\langle \E_m^{(\alpha)} \rangle$ and hence the $D_{mn}^{(\alpha)}$ can be determined from boundary data, but not the constants 
$A_{mn}^{(\alpha)}$.

Using the splitting method, we can split 
$$\langle \E_k \cdot \J_l \rangle = \langle \chi_1 \E_k \cdot \J_l \rangle + \langle \chi_2 \E_k \cdot \J_l \rangle
=\sigma^{(1)}A_{kl}^{(1)}+\sigma^{(2)}A_{kl}^{(2)}.$$ 
Notice that in contrast to the complex conductivity case, 
we have $\langle \E_k \cdot \J_l \rangle = \langle \E_l \cdot \J_k \rangle$. This gives us the following linear system:
\begin{equation}\label{powersystem}
\begin{bmatrix} 
\sigma^{(1)} & \sigma^{(2)} & 0 & 0 & 0 & 0 \\[0.1cm]
0 & 0 & \sigma^{(1)} & \sigma^{(2)} & 0 & 0 \\[0.1cm]
0 & 0 & 0 & 0 & \sigma^{(1)} & \sigma^{(2)}
\end{bmatrix}
\begin{bmatrix}  A_{11}^{(1)} \\[0.1cm] A_{11}^{(2)} \\[0.1cm] A_{21}^{(1)} \\[0.1cm] A_{21}^{(2)} \\[0.1cm] A_{22}^{(1)} \\[0.1cm] A_{22}^{(2)} \end{bmatrix}
=
\begin{bmatrix} \langle \E_1\cdot\J_1\rangle \\[0.1cm] \langle\E_1\cdot \J_2\rangle \\[0.1cm] \langle \E_2 \cdot \J_2 \rangle \end{bmatrix} .
\end{equation}
We pick our free variables $x^{(1)} := A_{11}^{(1)}, y^{(1)} := A_{21}^{(1)}, z^{(1)} := A_{22}^{(1)}$. Let $x^{(2)} = A_{11}^{(2)}, y^{(2)} = A_{21}^{(2)}, z^{(2)} = A_{22}^{(2)}$, then from our linear system we find 
\begin{align} \label{Aeqns}
\begin{split}
x^{(2)} & = (\langle \E_1\cdot\J_1\rangle - \sigma^{(1)} x^{(1)})/\sigma^{(2)}, \\
y^{(2)} & = (\langle \E_1\cdot\J_2\rangle - \sigma^{(1)} y^{(1)})/\sigma^{(2)}, \\
z^{(2)} & = (\langle \E_2\cdot\J_2\rangle - \sigma^{(1)} z^{(1)})/\sigma^{(2)}.
\end{split}
\end{align}
This, of course, has a unique solution if and only if $\sigma^{(2)} \neq 0$. We can assume that at least one of the phases has a nonzero conductivity (otherwise, we have nothing interesting to say) and then assign the label 2 to a phase with nonzero conductivity to ensure that $\sigma^{(2)} \neq 0$. If \eq{0.2} is satisfied, then the inequalities
$x^{(\alpha)}, y^{(\alpha)}, z^{(\alpha)} \leq (\ca)^2 f_{\Ga}$ must hold for $\Ga=1,2$, and these inequalities with \eq{Aeqns} define a rectangular prism in $(x^{(1)}, y^{(1)}, z^{(1)})$ space, which we 
call the compatible region. 

\begin{remark}
        This analysis shows one can apply the techniques in the paper of Milton and Thaler \cite{Thaler:2015:BVI} to bound the volume fractions of the phases in the case of real conductivity:
the volume fractions must be such that there is a nonempty feasible region of values of $(x^{(1)},y^{(1)},z^{(1)})$ such that the matrices $S^{(1)}$ and $S^{(2)}$ are both
positive semidefinite. We do not explore this further here.
\end{remark}

\begin{remark}
Before proceeding, we note that now, the matrix $\BS^{(\alpha)}$ can be written as 
$$\BS^{(\alpha)} = 
\left [
\begin{matrix}
x^{(\alpha)} - \dfrac{D_{11}^{(\alpha)}}{f_{\Ga}} && y^{(\alpha)} - \dfrac{D_{12}^{(\alpha)}}{f_{\Ga}} \\
y^{(\alpha)} - \dfrac{D_{21}^{(\alpha)}}{f_{\Ga}} && z^{(\alpha)} - \dfrac{D_{22}^{(\alpha)}}{f_{\Ga}} 
\end{matrix}
\right ].$$ 
\end{remark}

The values of $x^{(1)}$, $y^{(1)}$ and $z^{(1)}$ for which $\BS^{(1)}$ and $\BS^{(2)}$ are both positive semidefinite define what we call the feasible region in 
$(x^{(1)}, y^{(1)}, z^{(1)})$ space. 

\subsection{Bounds in Three Dimensions}
If the feasible region is empty or does not contain the compatible region then the material must have broken down for at least one of the two boundary conditions. 
This condition is however not so easy to check without plotting the regions in  $(x^{(1)}, y^{(1)}, z^{(1)})$ space, so let us now seek simpler algrebraic conditions
(which may however not be as tight). 
The fact that the $\BS^{(\alpha)}$ are positive semi-definite if breakdown has not occurred imposes the following conditions:
\begin{equation}\label{tr11}
x^{(\alpha)} - \dfrac{D_{11}^{(\alpha)}}{f_{\Ga}} \geq 0 \Rightarrow 	(\ca)^2 (f_{\Ga})^2 \geq D_{11}^{(\alpha)},
\end{equation}
\begin{equation}\label{tr22}
z^{(\alpha)} - \dfrac{D_{22}^{(\alpha)}}{f_{\Ga}} \geq 0 \Rightarrow 	(\ca)^2 (f_{\Ga})^2 \geq D_{22}^{(\alpha)},
\end{equation}
\begin{align}\label{detBound}
\det S^{(\alpha)} \geq 0 & \Leftrightarrow \left(x^{(\alpha)} - \dfrac{D_{11}^{(\alpha)}}{f_{\Ga}}\right)\left(z^{(\alpha)} - \dfrac{D_{22}^{(\alpha)}}{f_{\Ga}} \right) \geq \left(y^{(\alpha)} - \dfrac{D_{12}^{(\alpha)}}{f_{\Ga}}\right)^2 \nonumber \\
& \Rightarrow \left((\ca)^2 f_{\Ga} - \dfrac{D_{11}^{(\alpha)}}{f_{\Ga}}\right)\left((\ca)^2 f_{\Ga} - \dfrac{D_{22}^{(\alpha)}}{f_{\Ga}} \right) \geq \left(y^{(\alpha)} - \dfrac{D_{12}^{(\alpha)}}{f_{\Ga}}\right)^2,
\end{align}
where we have used the fact that $x^{(\alpha)}, z^{(\alpha)} \leq (\ca)^2 f_{\Ga}$.

The inequalities (\ref{tr11}) and (\ref{tr22}) give us elementary upper bounds on $D_{11}^{(\Ga)}$ and $D_{22}^{(\Ga)}.$
A potentially sharper bound on how close the $D_{nn}^{(\alpha)}$ can approach these elementary upper bounds is given by (\ref{detBound}), depending on the value of the right hand side. The $y^{(\alpha)}$ are unknown, which presents a problem. There is a way to deal with this: sum (\ref{detBound}) over $\alpha$ (possibly multiplying each equation by a 
positive weight $w^{(\alpha)}$) and substitute in the expression for $y^{(2)}$ given in (\ref{Aeqns}). Take the minimum of the right hand side, when treated as a quadratic with respect to $y^{(1)}$, which is likely to be nonzero. 

We have shown the following theorem.
\begin{theorem}
	Suppose one of the phases has nonzero conductivity; let it be phase 2, so $\sigma^{(2)} \neq 0$. Suppose neither of the volume fractions $f_{\Ga}$ are zero, and that both are known. Define $x^{(\alpha)}, y^{(\alpha)}, z^{(\alpha)}$ as before. Suppose further that the material has not broken down. Then 
the feasible region must be nonempty and intersect the compatible rectangular prism region and \emph{(\ref{Aeqns})-(\ref{detBound})} are satisfied.
\end{theorem}

\subsection{Improved Bounds in Two Dimensions}
\label{sec:ImprovedBounds}
In \cite{Thaler:2015:BVI}, having knowledge of two additional null Lagrangians gives an improved bound on the volume fraction. 
This is also the case in our situation; the additional null Lagrangians allow us to obtain a sharper bound than in (\ref{detBound}). The null-Lagrangians, which can be determined from boundary measurements, are $\BE_1 \cdot \BR_\perp \BE_2$ and $\BJ_1 \cdot \BR_\perp \BJ_2$ where 
\beq \BR_\perp=\begin{bmatrix} 0 & 1 \\ -1 & 0 \end{bmatrix}
\eeq{2.18}
denotes the matrix for a $90\,^{\circ}$ clockwise rotation. Assuming $|\sigma^{(1)}| \ne |\sigma^{(2)}|$ they  show the quantities $B_{12}^{(\Ga)}\equiv\lang\BE_1^{(\Ga)}\cdot\BR_\perp\BE_2^{(\Ga)}\rang$ can be expressed in terms of these null-Lagrangians through the identity
\beq
  \begin{bmatrix} 
    B^{(1)}_{12} \\[0.1cm] B^{(2)}_{12} 
  \end{bmatrix}
  = \dfrac{1}{|\sigma^{(2)}|^2 -
    |\sigma^{(1)}|^2}
  \begin{bmatrix} 
    |\sigma^{(2)}|^2 \langle \BE_1 \cdot \BR_\perp \BE_2 \rangle -
    \langle \BJ_1 \cdot \BR_\perp \BJ_2 \rangle \\[0.1cm]
    -|\sigma^{(1)}|^2 \langle \BE_1 \cdot \BR_\perp \BE_2 \rangle +
    \langle \BJ_1 \cdot \BR_\perp \BJ_2 \rangle.
  \end{bmatrix} 
\eeq{2.18a}
The paper \cite{Thaler:2015:BVI} considers the following quantity. For  $\mathbf{c}^{(\alpha)}, \mathbf{d}^{(\alpha)}$ in $\mathbb{R}^2$ and for $\alpha = 1, 2$, define
\begin{align}\label{h}
\ha(\x;\mathbf{c}^{(\alpha)}, \mathbf{d}^{(\alpha)}) := \displaystyle\sum_{m=1}^{2} & c^{(\alpha)}_m \left[\E^{(\alpha)}_m(\x) - \dfrac{\chi_\alpha(\x)}{f_{\Ga}}\langle\E_m^{(\alpha)}\rangle\right] \\ & +  \displaystyle\sum_{n=1}^{2} d^{(\alpha)}_n \left[R_{\perp}\E^{(\alpha)}_n(\x) - \dfrac{\chi_\alpha(\x)}{f_{\Ga}}\langle R_{\perp} \E^{(\alpha)}_n\rangle\right].
\end{align}
From the positivity of the variance $\langle \ha \cdot \ha \rangle \geq 0$, for all $\mathbf{c}^{(\alpha)}, \mathbf{d}^{(\alpha)}$ in $\mathbb{R}^2$,  
they derive the improved bounds
\[
\det[\BS^{(1)}]\geq \Gj^{(1)}\quad{\rm and}\quad \det[\BS^{(2)}]\geq \Gj^{(2)},
\]
where
\beq
\Gj^{(\alpha)} := \left[B_{12}^{(\Ga)} - \frac{1}{f^{(\alpha)}}\langle\BE_1^{(\alpha)}\rangle \cdot R_{\perp} \langle\BE_2^{(\alpha)}\rangle\right]^2 \ge 0.
\eeq{2.19}
can be determined from boundary measurements. These improved bounds imply a reduced feasible region in  $(x^{(1)}, y^{(1)}, z^{(1)})$ space
and imply
\begin{align}\label{lgrnBound}
\left((\ca)^2 f_{\Ga} - \dfrac{D_{11}^{(\alpha)}}{f_{\Ga}}\right)\left((\ca)^2 f_{\Ga} - \dfrac{D_{22}^{(\alpha)}}{f_{\Ga}} \right) \geq \left(y^{(\alpha)} - \dfrac{D_{12}^{(\alpha)}}{f_{\Ga}}\right)^2 + \tau_f^{(\alpha)},
\end{align}
where we have again used the fact that $x^{(\alpha)}, z^{(\alpha)} \leq (\ca)^2 f_{\Ga}$.

Recall the issue with the $y^{(\alpha)}$ being unknown. Equation (\ref{lgrnBound}) allows us to bypass the $y^{(\alpha)}$ altogether, and obtain
\begin{align}\label{lgrnBound1}
\left((\ca)^2 f_{\Ga} - \dfrac{D_{11}^{(\alpha)}}{f_{\Ga}}\right)\left((\ca)^2 f_{\Ga} - \dfrac{D_{22}^{(\alpha)}}{f_{\Ga}} \right) \geq \tau_f^{(\alpha)},
\end{align}
which gives us separate inequalities that constrain how close the $D_{nn}^{(\Ga)}$ can approach their elementary upper bounds. Alternately, we can again take a weighted sum
of \eq{lgrnBound} over $\alpha$ and minimize the quadratic on the right hand side, and derive a sharper bound than in (\ref{lgrnBound1}).

\begin{theorem}
	Suppose, as before, $\sigma^{(2)} \neq 0,$ neither volume fractions $f_{\Ga}$ are zero, and both are known. Define $x^{(\alpha)}, y^{(\alpha)}, z^{(\alpha)}$ as before, and suppose the two-dimensional material has not broken down. Then, the reduced feasible region must be nonempty and intersect the compatible rectangular prism region, and \emph{(\ref{Aeqns})-(\ref{tr22}),(\ref{lgrnBound}),(\ref{lgrnBound1})} are satisfied.
\end{theorem}

\begin{remark}
	We suspect that although there will be geometries where $\tau_f^{(\alpha)} = 0$, it will usually be nonzero.
\end{remark}

\section{Complex Conductivity}
\setcounter{equation}{0}
\subsection{Simple Conditions on the local field for the onset of nonlinearities}

The electromagnetic response of a body to oscillating fields of a given frequency $\Go$ is well 
described by the quasistatic equations when the relevant wavelengths and attenuation lengths are large compared to the body. The quasistatic
equations, in a locally isotropic body, are just like those for static conductivity
\beq \BJ(\Bx)=\Gs(\Bx)\BE(\Bx), \quad \Div\BJ=0,\quad \BE=-\Grad V, \eeq{2.1}
except the conductivity $\Gs(\Bx)$, and the fields $\BJ$, $\BE$ and $V$ are complex:
\beq \Gs=\Gs_1+i\Gs_2, \quad \BJ=\BJ_1+i\BJ_2, \quad \BE=\BE_1+i\BE_2, \quad V=V_1+iV_2, \eeq{2.2}
where a subscript $1$ denotes the real part, while a subscript $2$ denotes the imaginary part. 
The local physical electric field is
\beq \Be(\Bx,t)=\Real\{(\BE_1(\Bx)+i\BE_2(\Bx))e^{-i\omega t}\}, \eeq{2.3}
where $t$ is the time. In general as the time $t$ varies the endpoint of the vector $\Be(\Bx,t)$ describes an ellipse (in the plane
spanned by $\BE_1$ and $\BE_2$). Let us assume that the onset of nonlinearities at a point $\Bx$ just depends on the value the
electric field $\Be(\Bx,t)$ takes at the point $\Bx$ as time varies. Since the condition must be independent of how we choose the origin of time, 
the condition must only depend on the invariants, namely the lengths of the minor and major axes of the ellipse. If at our point $\Bx$ the major axis happened to coincide
with the value $\BE_1(\Bx)$ of $\Be(\Bx,t)$ at $t=0$ then its easy to check that $\BE_2(\Bx)$ is the minor axis of the ellipse, and since these axes are perpendicular
$\BE_1(\Bx)\cdot\BE_2(\Bx)=0$. In this case our invariants can be taken as $|\BE_1(\Bx)|$ and $|\BE_2(\Bx)|$. More generally, if we choose a different origin of time $t_0$
(which could depend on $\Bx$), then the physical electric field is
\beq \Be(\Bx,t)=\Real\{(\BE_1'(\Bx)+i\BE_2'(\Bx))e^{-i\omega (t-t_0)}\}, \eeq{2.4}
where
\beq \BE_1'(\Bx)+i\BE_2'(\Bx)=e^{-i\omega t_0}(\BE_1(\Bx)+i\BE_2(\Bx)). \eeq{2.5}
Thus we have the identification
\beq \BE_1'(\Bx)=[\cos(\Go t_0)\BE_1(\Bx)+\sin(\Go t_0)\BE_2(\Bx)], \quad \BE_2'(\Bx)=[\cos(\Go t_0)\BE_2(\Bx)-\sin(\Go t_0)\BE_1(\Bx)], \eeq{2.6}
and 
\beq \BE_1'(\Bx)\cdot\BE_2'(\Bx)=\cos(2\Go t_0)[\BE_1(\Bx)\cdot\BE_2(\Bx)]-\sin(2\Go t_0)[|\BE_1(\Bx)|^2-|\BE_2(\Bx)|^2]/2 \eeq{2.7}
is zero when
\beq \tan(2\Go t_0)=\frac{2\BE_1(\Bx)\cdot\BE_2(\Bx)}{|\BE_1(\Bx)|^2-|\BE_2(\Bx)|^2}, \eeq{2.8}
or when $2\Go t_0=\pi/2$ if $|\BE_1(\Bx)|^2=|\BE_2(\Bx)|^2$. With $t_0$ chosen in this way, the invariants which are the axes of the ellipse 
can be taken as $|\BE'_1(\Bx)|$ and $|\BE'_2(\Bx)|$. 

\begin{remark}
	In summary, at each point $\Bx$ the endpoint of the physical electric field vector $\Be(\Bx,t)=\Real\{(\BE_1(\Bx)+i\BE_2(\Bx))e^{-i\omega t}\}$ 
describes at ellipse in the plane spanned by $\BE_1(\Bx)$ and $\BE_2(\Bx)$ with axes  $|\BE'_1(\Bx)|$ and $|\BE'_2(\Bx)|$, where $\BE'_1(\Bx)$ and
$\BE'_2(\Bx)$ are given by \eq{2.6} and $t_0$ is given by \eq{2.8}.
\end{remark}

\begin{remark}
	The condition for the onset of nonlinearities for at the point $\Bx$ in phase $\alpha$, if local,  must just depend only on the local invariants of the field,
namely the ellipse axes. Thus the condition for the onset of nonlinearities at the point $\Bx$ in phase $\Ga$ can be expressed in the form
\beq F^{(\Ga)}(|\BE'_1(\Bx)|,|\BE'_2(\Bx)|)\geq 0, \eeq{2.9}
for some function $F^{(\Ga)}$ which is symmetric in its arguments. We will only consider the simple case where
\beq F^{(\Ga)}(|\BE'_1(\Bx)|,|\BE'_2(\Bx)|)=|\BE'_1(\Bx)|^2+|\BE'_2(\Bx)|^2-(c^{(\Ga)})^2=|\BE_1(\Bx)|^2+|\BE_2(\Bx)|^2-(c^{(\Ga)})^2, \eeq{2.10}
for some positive constants $c^{(\Ga)}$ (where the last identity in \eq{2.10} follows by taking the modulus of both sides of \eq{2.5}). The motivation for considering
such a criterion is not just for simplicity, but also because the intensity $I(\Bx)=|\BE_1(\Bx)|^2+|\BE_2(\Bx)|^2$ is proportional the time averaged dissipation of electrical power into 
heat, and it makes physical sense that the materials may break down if this is too high. To see this, note that the physical electric and current fields at the point $\Bx$
can be expressed as 
\beqa \Be(\Bx,t) & = &\Real\{(\BE_1(\Bx)+i\BE_2(\Bx))e^{-i\omega t}\} \nonum
& = & [(\BE_1(\Bx)+i\BE_2(\Bx))e^{-i\omega t}+(\BE_1(\Bx)-i\BE_2(\Bx))e^{+i\omega t}]/2, \nonum
 \Bj(\Bx,t)& = & \Real\{(\BJ_1(\Bx)+i\BJ_2(\Bx))e^{-i\omega t}\} \nonum
& = & [(\BJ_1(\Bx)+i\BJ_2(\Bx))e^{-i\omega t}+(\BJ_1(\Bx)-i\BJ_2(\Bx))e^{+i\omega t}]/2. \eeqa{2.10-0}
Their dot product $\Bj(\Bx,t)\cdot\Be(\Bx,t)$ represents the instantaneous electrical power density which is dissipated into heat. Averaging over time,
and using the fact that the time average of $e^{-2i\omega t}$ and $e^{+2i\omega t}$ is zero, we see that
\beqa \lang\Bj(\Bx,t)\cdot\Be(\Bx,t)\rang_t & = & [\BJ_1(\Bx)\cdot\BE_1(\Bx)+\BJ_2(\Bx)\cdot\BE_2(\Bx)]/2 \nonum
& = &\Gs_1[|\BE_1(\Bx)|^2+|\BE_2(\Bx)|^2]/2=\Gs_1I(\Bx)/2,
\eeqa{2.10-1}
in which $\lang\cdot\rang_t$ denotes a time average. 
Alternative criteria can also have merit from a physical viewpoint. For example, if the frequency $\Go$ is low the breakdown of materials might be dictated
by the peak strength of the electric field, in which case the criterion would be
\beq F^{(\Ga)}(|\BE'_1(\Bx)|,|\BE'_2(\Bx)|)=\max\{|\BE'_1(\Bx)|,|\BE'_2(\Bx)|\}-c^{(\Ga)}, \eeq{2.10a}
or it could be dictated by the peak value of the power dissipation into heat,
\beqa \Bj(\Bx,t)\cdot\Be(\Bx,t) & = &  [\Gs_1I(\Bx)+\Real\{(\BJ_1(\Bx)+i\BJ_2(\Bx))\cdot(\BE_1(\Bx)+i\BE_2(\Bx))e^{-2i\omega t}\}]/2,\nonum
&~&
\eeqa{2.10-2}
in which case the criterion would be
\beqa &~& F^{(\Ga)}(|\BE'_1(\Bx)|,|\BE'_2(\Bx)|) \nonum
&~&\quad  =\Gs_1^{(\Ga)}[|\BE'_1(\Bx)|^2+|\BE'_2(\Bx)|^2]/2+\sqrt{(\Gs_1^{(\Ga)})^2+(\Gs_2^{(\Ga)})^2}||\BE'_1(\Bx)|^2-|\BE'_2(\Bx)|^2|/2-(c^{(\Ga)})^2,\nonum
&~&
\eeqa{2.10-3}
in which $\Gs_1^{(\Ga)}$ and $\Gs_2^{(\Ga)}$ are the values of $\Gs_1(\Bx)$ and $\Gs_2(\Bx)$ in phase $\alpha$.
We will not consider these criteria further, as they are more difficult to treat than the criteria \eq{2.10}. 
\end{remark}

\subsection{Conditions from boundary measurements which guarantee nonlinearities are present}

Consider a two phase isotropic material, in two or three dimensions, with complex conductivity 
\beq
\sigma^{(\alpha)}=\sigma_1^{(\alpha)}+i\sigma_2^{(\alpha)},
\eeq{2.11}
where the superscript $\alpha=1,2$ denotes the phase and the subscript denotes the real and imaginary component of the conductivity. 
In general the complex conductivities $\sigma^{(1)}$ and $\sigma^{(2)}$ depend on the frequency $\Go$. It is helpful
to also introduce the fields
\beq
\BE^{(\alpha)}(\Bx)=\Gc_{\Ga}(\Bx)\BE(\Bx)=(\BE_1^{(\alpha)}(\Bx)+i\BE_2^{(\alpha)}(\Bx)),
\eeq{2.12}
where $\Gc_{\Ga}(\Bx)$ is the characteristic function taking the value $1$ in phase $\Ga$ and zero elsewhere. Our simplified condition for the onset on nonlinearities 
at point $\Bx$ in phase $\Ga$ is given by
\beq |\BE_1^{(\alpha)}(\Bx)|^2+|\BE_2^{(\alpha)}(\Bx)|^2  \geq (c^{(\Ga)})^2. \eeq{2.13}
If this condition is met, then we say that our material has become nonlinear. As observed in \cite{Thaler:2015:BVI} the quantities
\beq \langle\BE_i^{(\Ga)}\rangle,\quad \langle\BE_i\cdot\BJ_k\rangle \eeq{2.14}
 can be determined by  boundary measurements for all $\Ga$, $i$, and $k$. Employing the splitting method as described earlier, the six quantities 
\beq  A_{mn}^{(\alpha)} =
    \langle\BE_m^{(\alpha)}\cdot\BE_n^{(\alpha)}\rangle \quad
    (\text{for} \ \alpha, \ m, \ n = 1, \ 2)
\eeq{2.15}
 are related by four equations, which when solved give
\beq \begin{bmatrix}
    A_{21}^{(1)} \\[0.1cm] A_{21}^{(2)} \\[0.1cm] A_{22}^{(1)}
    \\[0.1cm] A_{22}^{(2)} 
  \end{bmatrix}
  =  \begin{bmatrix} 
    -\sigma_2^{(1)} & -\sigma_2^{(2)} & 0 & 0 \\[0.1cm] \sigma_1^{(1)}
    & \sigma_1^{(2)} & 0 & 0\\[0.1cm] \sigma_1^{(1)} & \sigma_1^{(2)}
    & -\sigma_2^{(1)} & -\sigma_2^{(2)} \\[0.1cm] \sigma_2^{(1)} &
    \sigma_2^{(2)} & \sigma_1^{(1)} & \sigma_1^{(2)}
  \end{bmatrix}^{-1}
\begin{bmatrix} 
    \lla \BE_1\cdot\BJ_1\rra - \sigma_1^{(1)} x - \sigma_1^{(2)} y
    \\[0.1cm] \lla \BE_1\cdot \BJ_2\rra -\sigma_2^{(1)} x -
    \sigma_2^{(2)} y \\[0.1cm] \lla \BE_2\cdot \BJ_1 \rra \\[0.1cm] \lla
    \BE_2 \cdot \BJ_2 \rra 
  \end{bmatrix}.
\eeq{2.16}
in terms of the ``free variables'' $x\equiv A_{11}^{(1)}$ and $y \equiv A_{11}^{(2)}$ 
(assuming $\beta = \sigma_1^{(1)}\sigma_2^{(2)} - \sigma_2^{(1)}\sigma_1^{(2)}\ne 0$). These two free variables $x$ and $y$ cannot however be directly evaluated from boundary
measurements if data are only available at one frequency. If the materials have a linear response everywhere then we have for $\Ga=1,2$,
\beq  A_{11}^{(\Ga)}+A_{22}^{(\Ga)}=\langle \BE_1^{(\alpha)}\cdot\BE_1^{(\alpha)}+\BE_2^{(\alpha)}\cdot\BE_2^{(\alpha)}\rangle\leq f^{\Ga}(c^{(\Ga)})^2, \eeq{2.17}
and by using \eq{2.16} to eliminate $A_{22}^{(1)}$ and $A_{22}^{(2)}$
each of these conditions reduces to a linear inequality in the $(x,y)$ plane. The intersection of the two linear inequalities defines what we call the compatible region in the $(x,y)$ plane. 
 
Following the procedure of Thaler and Milton \cite{Thaler:2015:BVI}, we can use the positivity of the variance, $\langle \ga \cdot \ga \rangle \geq 0$, for all $\mathbf{c}^{(\alpha)} \in \mathbb{R}^2$,
where $\ga$ is given by \eq{g}, to obtain the condition that the matrices $\BS^{\Ga}$ given by \eq{0.6} are positive semidefinite. Making the substitutions \eq{2.16}
and the symmetric matrices $\BS^\Ga$ can be expressed in terms of $x$ and $y$:
\beq
 \begin{aligned}
 \BS^{(1)}(x,y) &:= \begin{bmatrix} x - \dfrac{\|\langle\Eoo\rangle\|^2}{f^{(1)}} & S_{21}^{(1)}(x,y) \\ S_{21}^{(1)}(x,y) & -x + \etaone - \dfrac{\|\langle\Eto\rangle\|^2}{f^{(1)}} \end{bmatrix}, \\[0.1cm]
 \BS^{(2)}(x,y) &:= \begin{bmatrix} y - \dfrac{\|\langle\Eot\rangle\|^2}{f^{(2)}} & S_{21}^{(2)}(x,y) \\ S_{21}^{(2)}(x,y) & -y + \etatwo - \dfrac{\|\langle\Ett\rangle\|^2}{f^{(2)}} \end{bmatrix}, \\
 \end{aligned} 
\eeq{2.17a}
where 
 \[
 \begin{aligned}
 S_{21}^{(1)}&(x,y) = -\gamma x - \psione y + \xione - \dfrac{\langle\Eoo\rangle\cdot\langle\Eto\rangle}{f^{(1)}} ;\\
 S_{21}^{(2)}&(x,y) =  \psitwo x + \gamma y  - \xitwo - \dfrac{\langle\Eot\rangle\cdot\langle\Ett\rangle}{f^{(2)}} ; \\
 \beta &= \sigma_1^{(1)}\sigma_2^{(2)} - \sigma_2^{(1)}\sigma_1^{(2)} ; \quad
 \gamma = \frac{\sigma_1^{(1)} \sigma_1^{(2)} + \sigma_2^{(1)} \sigma_2^{(2)}}{\beta} ; \quad
 \psione = \frac{\left|\sigma^{(2)}\right|^2}{\beta} ; \quad 
 \psitwo = \frac{\left|\sigma^{(1)}\right|^2}{\beta} ;\\
 \xione  &= \dfrac{\sigma_2^{(2)} \lla \BE_1\cdot\BJ_2\rra + \sigma_1^{(2)} \lla \BE_1\cdot \BJ_1\rra}{\beta} ; \quad
 \xitwo = \dfrac{\sigma_2^{(1)} \lla \BE_1\cdot\BJ_2\rra + \sigma_1^{(1)} \lla\BE_1\cdot \BJ_1\rra}{\beta} ; \\
 \etaone &= \dfrac{\sigma_1^{(2)}\left(\lla\BE_2\cdot\BJ_1\rra - \lla\BE_1\cdot\BJ_2\rra\right) + \sigma_2^{(2)}\left(\lla\BE_1\cdot\BJ_1\rra + \lla\BE_2\cdot\BJ_2\rra\right)}{\beta} ; \\
 \etatwo &= \dfrac{\sigma_1^{(1)} \left(\lla\BE_1\cdot\BJ_2\rra - \lla\BE_2\cdot\BJ_1\rra\right) - \sigma_2^{(1)}\left(\lla\BE_1\cdot\BJ_1\rra + \lla\BE_2\cdot\BJ_2\rra\right)}{\beta} .
 \end{aligned}
\]

The constraint that the matrices $\BS^{(\alpha)}$ must be positive semidefinite confines the pair $(x,y)$ to lie within a region which is the intersection of the ellipse
$\det [\BS^{(1)}(x,y)]\geq 0$ with the ellipse $\det [\BS^{(2)}(x,y)]\geq 0$. We call this region of intersection the feasible region. If it is empty, or does not intersect the
compatible region, then one or both of the materials must have become nonlinear somewhere (see Figure 2).

We have outlined the proof of the following theorem 

\begin{theorem}
	Suppose that $\beta = \sigma_1^{(1)}\sigma_2^{(2)} - \sigma_2^{(1)}\sigma_1^{(2)}\ne 0$, and that the volume fractions $f_{\Ga}$ are both nonzero and known. Defining $x$ and $y$ as above, if the material has not experienced the onset of nonlinearities, then it is necessary that the region of intersection of the two 
ellipses in the $x-y$ plane given by the constraints on $\BS^{(\alpha)}$ must be nonempty and intersect the compatible region.
\end{theorem}

\begin{figure}[htp]
	\centering
	\begin{subfigure}[b]{0.3\textwidth}
		\includegraphics[width=1.3\textwidth]{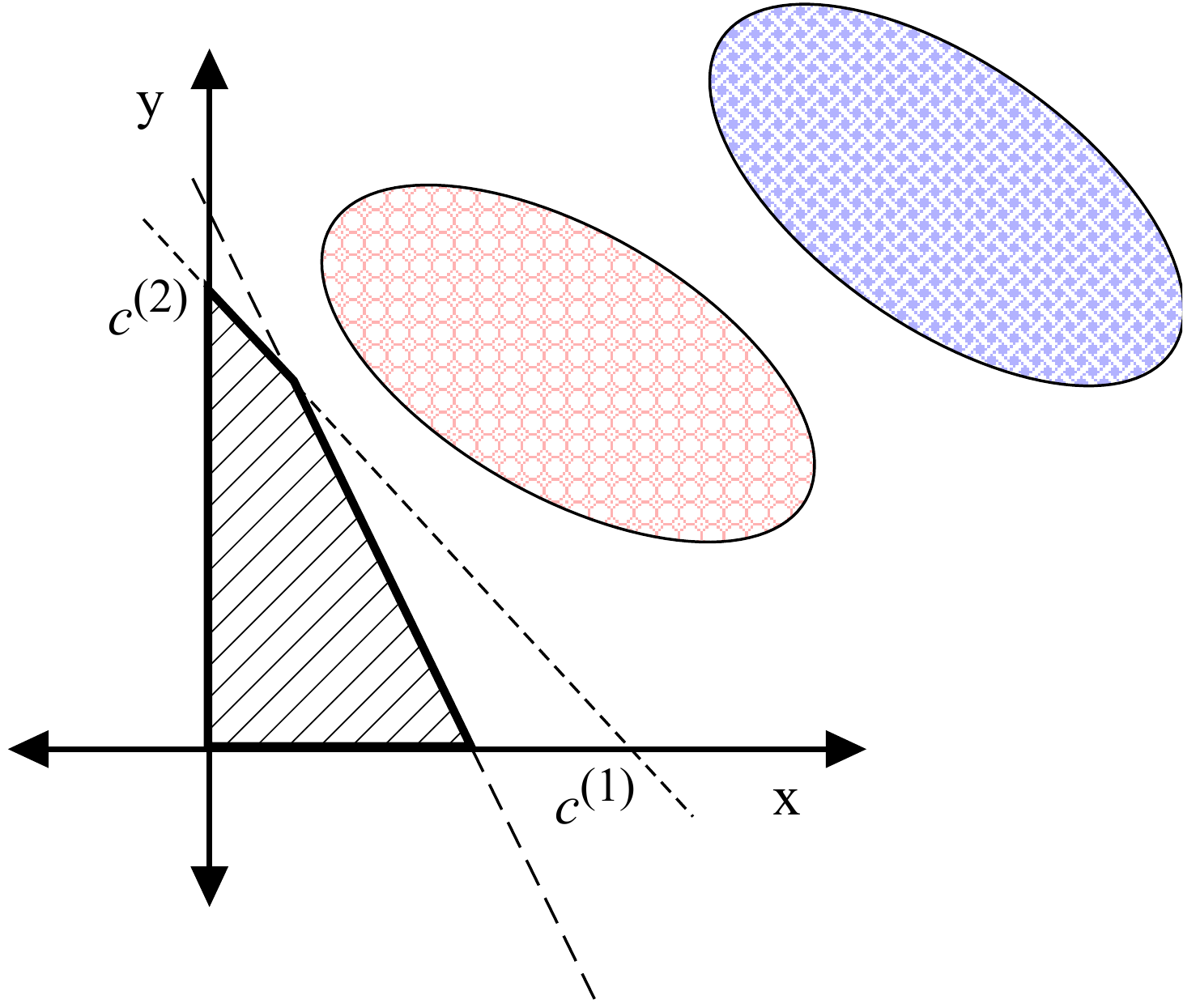}
		\caption{}
	\end{subfigure}\hspace{60 pt}
		\begin{subfigure}[b]{0.3\textwidth}
			\includegraphics[width=1\textwidth]{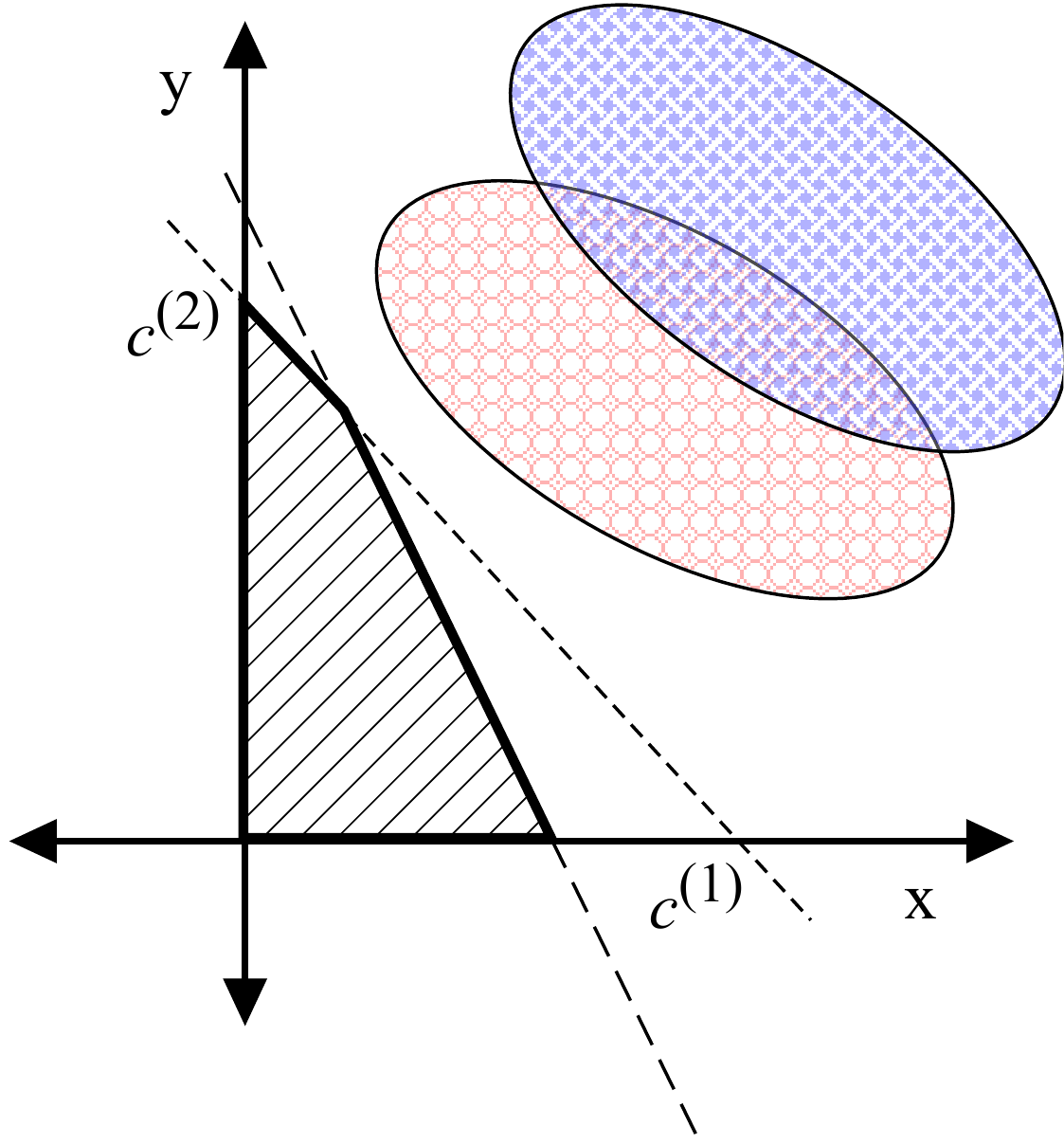}
			\caption{}
		\end{subfigure}
		\caption{The material must have experienced nonlinearities if either (a) the feasible region (which is the region of intersection 
of the two ellipses) is empty or (b) The feasible region does not intersect the compatible region (which is the polygonal region shaded with black lines).}.
	\label{fig:Guaranteed}
\end{figure}

\begin{figure}[htp]
	\centering
	\includegraphics[width=0.50\textwidth]{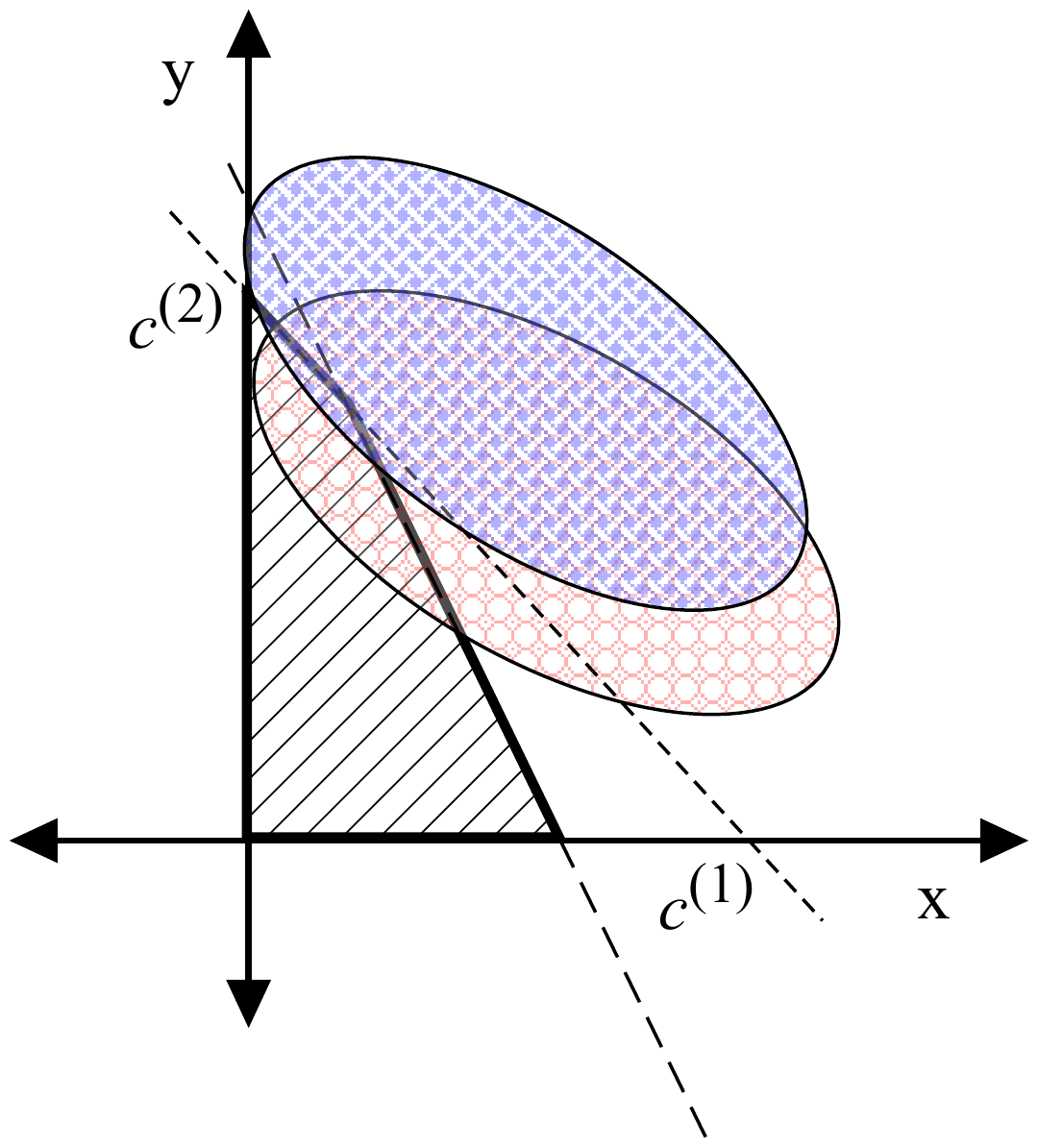}
	\caption{Here the feasible region intersects the compatible region, so the material may, or may not, have experienced the onset of nonlinearities depending on the internal geometry of the body.}
	\label{fig:Possible}
\end{figure}

Using the additional null-Lagrangians $\BE_1 \cdot \BR_\perp \BE_2$ and $\BJ_1 \cdot \BR_\perp \BJ_2$, as described in section \ref{sec:ImprovedBounds} and as in reference \cite{Thaler:2015:BVI}, 
the feasible region is reduced to the region in the $(x,y)$ plane which is the intersection of the
two ellipses  
\begin{equation}
\det[\BS^{(1)}(x,y)]\geq \Gj^{(1)}\quad{\rm and}\quad \det[\BS^{(2)}(x,y)]\geq \Gj^{(2)},
\label{eq:TighterBounds}
\end{equation}
where $\Gj^{(\alpha)}$ is given by \eq{2.19}. Here we note that it is necessary but not sufficient to say that if the material has not experienced the onset of nonlinearities then the feasible region must be nonempty and intersect the compatible region. 

\begin{theorem}
	Suppose that $\beta = \sigma_1^{(1)}\sigma_2^{(2)} - \sigma_2^{(1)}\sigma_1^{(2)}\ne 0$, and that the volume fractions $f^{(\Ga)}$ are both nonzero and known. Defining $x$ and $y$ as above, if the material has not experienced the onset of nonlinearities, then it is necessary that the region of intersection (in the $x-y$ plane) of the two ellipses given by (\ref{eq:TighterBounds}) is nonempty and intersects the compatible region given by \eq{2.17}. 
\end{theorem}

\begin{figure}[htp]
	\centering
	\includegraphics[width=0.50\textwidth]{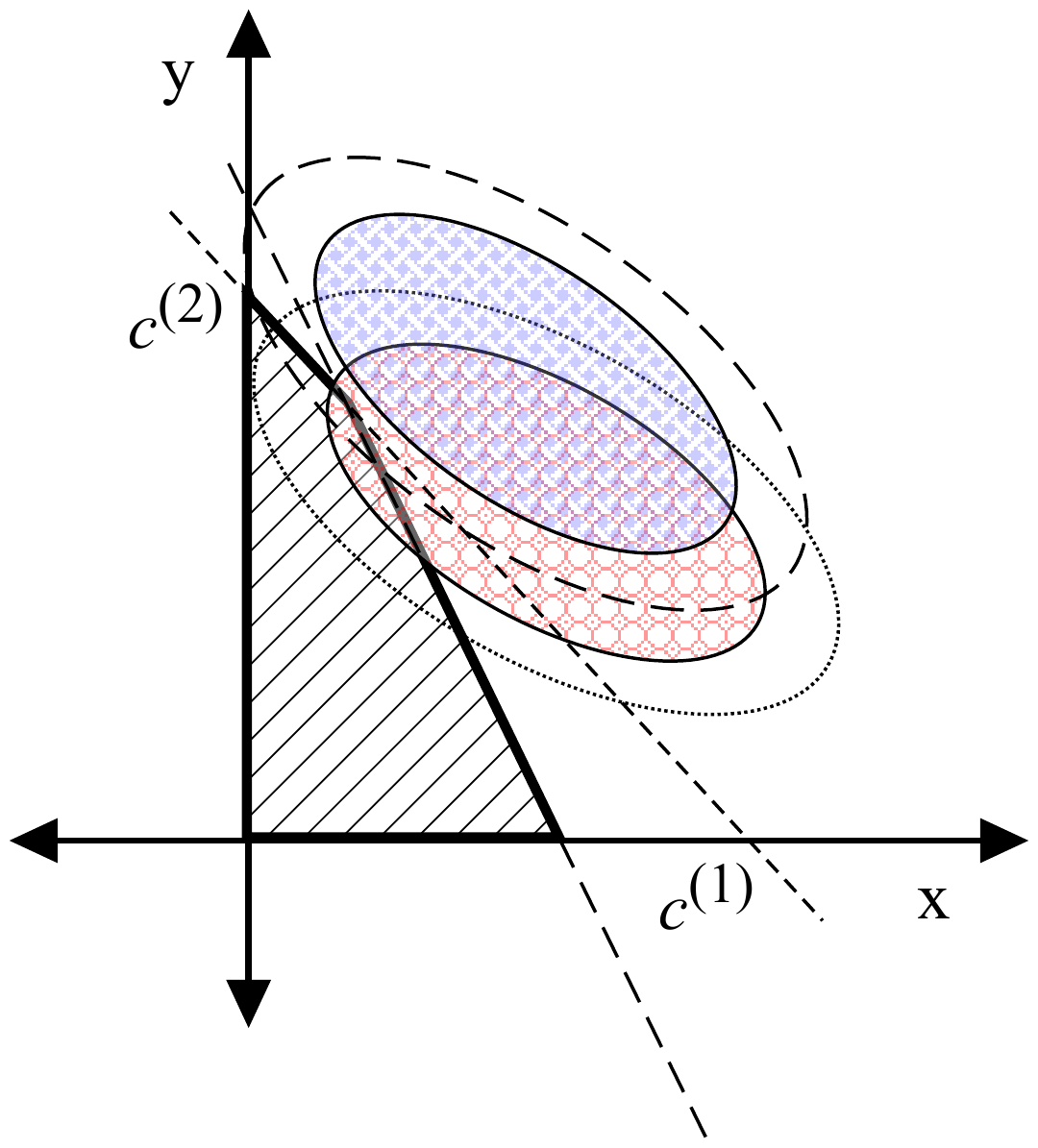}
	\caption{The dashed and dotted elliptical boundaries mark the previous bounds as in figure \ref{fig:Possible}, while the solid shaded portions show the improved bounds. The intersection of the compatible region and the feasible region is now empty so it is guaranteed that the material has experienced the onset of nonlinearities.}
	\label{fig:Improved}
\end{figure}

\section{Two-dimensional Elasticity}
\setcounter{equation}{0}
In two-dimensional linear elasticity (applicable to plane stress or plane strain problems) with isotropic constituents the constitutive equation takes the form
\beq \BGj=\BC\Grad\Bu=\Gm[\Grad\Bu+(\Grad\Bu)^T]+(\Gk-\Gm)\Tr(\Grad\Bu)\BI,
\eeq{5.1}
where $\BGj(\Bx)$ is the $2\times 2$ symmetric matrix valued stress, $\BC(\Bx)$ is the elasticity tensor,
$\Grad\Bu(\Bx)$ is the $2\times 2$ matrix valued displacement field gradient, and $\Gm(\Bx)$ and $\Gk(\Bx)$ are the local shear and bulk moduli. On the space of $2\times 2$ matrices it is convenient to introduce the basis
\beq B=\frac{1}{\sqrt{2}}(\begin{bmatrix} 0 & 1 \\ -1 & 0 \end{bmatrix}, \begin{bmatrix} 1 & 0 \\ 0 & 1 \end{bmatrix}, \begin{bmatrix} 1 & 0 \\ 0 & -1 \end{bmatrix}, \begin{bmatrix} 0 & 1 \\ 1 & 0 \end{bmatrix}). \eeq{5.11}
In this basis the stress $\BGj(\Bx)$ is represented by the vector field
\beq\BGj = ( 0, \tau_{1}, \tau_{2}, \tau_{3}), \eeq{5.12}
where the first element is zero because the stress matrix is symmetric. The displacement field gradient is represented by the vector field
\beq \nabla \Bu=(F_{0},\epsilon_{1},\epsilon_{2},\epsilon_{3}), \eeq{5.15}
where $F_{0}$ is proportional to the antisymmetric part of $\nabla\Bu$, corresponding to the local rotation, and $\epsilon_{1}$, $\epsilon_{2}$, and $\epsilon_{3}$ just depend on the
symmetric part of $\Grad\Bu$, which is the strain $\BGe(\Bx)=[\Grad\Bu+(\Grad\Bu)^T]/2$.

If the stress is too large, then nonlinear effects may become important. In particular if the stress is sufficiently large the material may undergo plastic yielding. For simplicity
we assume the response of the material is linear until it undergoes plastic yielding and we seek to determine boundary values of the displacement $\Bu$ and traction
$\BGj\Bn$ which if observed would necessarily imply that the material inside the body had yielded. (Without this assumption these boundary values would imply that the material is no
longer correctly modelled by the linear elasticity equations).
 
There are multiple yield criteria that have proven to be useful. In three dimensional elasticity these
criteria depend on the three eigenvalues $\Gs_1$, $\Gs_2$ and $\Gs_3$ of the $3\times 3$ symmetric matrix valued stress: if these eigenvalues lie
inside the yield surface the material will not have yielded, otherwise it will have yielded, or is at the threshold of yielding. 
Among the simplest models for the yield surface are the Von Mises Yield surface
\beq (\sigma_{1}-\sigma_{2})^{2}+(\sigma_{2}-\sigma_{3})^{2}+(\sigma_{3}-\sigma_{1})^{2}={\rm constant}, \eeq{5.3}
and the Tresca Yield Surface
\beq \max\{|\sigma_{1}-\sigma_{2}|,|\sigma_{2}-\sigma_{3}|,|\sigma_{3}-\sigma_{1}|\}={\rm constant}. \eeq{5.4}
We only consider two dimensional elasticity, so \eq{5.3} and \eq{5.4} reduce to
\beq |\sigma_{1}-\sigma_{2}|={\rm constant}, \eeq{5.5}
where $\Gs_1$ and $\Gs_2$ are the eigenvalues of the $2\times 2$ symmetric matrix valued stress $\BGj$.
By rotating the stress matrix $\BGj$ at a given point so it is diagonal and takes the form
\beqa \BGj & = &  \begin{bmatrix} \sigma_{1} & 0 \\ 0 & \sigma_{2} \end{bmatrix} \nonum
           & = &  \frac{\sigma_{1}+\sigma_{2}}{2} \begin{bmatrix} 1 & 0 \\ 0 & 1 \end{bmatrix} + \frac{\sigma_{1}-\sigma_{2}}{2} \begin{bmatrix} 1 & 0 \\ 0 & -1 \end{bmatrix}, \eeqa{5.13}
we see that (since $\tau_{3}=0$) 
\beq \frac{(\sigma_{1}-\sigma_{2})^{2}}{2}=\tau_{2}^{2}+\tau_{3}^{2}. \eeq{5.14}
As the right hand side remains invariant as the matrix $\BGj$ is rotated this expression is valid even if $\BGj$ is not diagonal
and so the Von Mises-Tresca criterion becomes 
\beq \tau_{2}^{2}+\tau_{3}^{2}={\rm constant}. \eeq{5.14a}

To determine conditions which necessarily imply yield has occurred, we will use the splitting method used by Milton and Nguyen
\cite{Milton:2011:BVF}. They note that the quantities
\beq E = \langle \BGj \cdot \nabla \Bu \rangle,\quad \BGj_{0}=\langle \BGj \rangle, \quad \langle \nabla \Bu \rangle, \quad a=\langle \det \BGj \rangle, \quad 
b=\langle \det \nabla \Bu \rangle, \eeq{5.10}
can all be evaluated from boundary measurements, using integration by parts. In the basis \eq{5.11} the expressions for $a$ and $b$ become
\beq a=\frac{1}{2}\langle \tau_{1}^{2}-\tau_{2}^{2}-\tau_{3}^{2} \rangle,\quad b=\frac{1}{2}\langle F_{0}^{2}+\epsilon_{1}^{2}-\epsilon_{2}^{2}-\epsilon_{3}^{2}\rangle. \eeq{5.16}
Since not much can be said about $\langle F_{0}^{2}\rangle$ other than it being not less than $\langle F_{0}\rangle^2$, it is useful to introduce the additional
quantity
\beqa c\equiv b-\frac{1}{2}\langle F_{0} \rangle ^{2} \geq b-\frac{1}{2}\langle F_{0}^{2} \rangle=\frac{1}{2}\langle \epsilon_{1}^{2}-\epsilon_{2}^{2}-\epsilon_{3}^{2}\rangle, \eeqa{5.17}
which can also be determined from boundary measurements.
The inequality here becomes an equality if and only if $F_{0}$ is constant everywhere.

The total elastic energy $E$ can be subdivided into separate quantities for each phase and according to whether it is a bulk or shear energy component:
\beqa E_{1b}=\langle \chi_{1}\tau_{1}\epsilon_{1} \rangle=2\kappa_{1}\langle\chi_{1}\epsilon_{1}^{2}\rangle, \eeqa{5.18}
\beqa E_{2b}=\langle \chi_{2}\tau_{1}\epsilon_{1} \rangle=2\kappa_{2}\langle\chi_{2}\epsilon_{1}^{2}\rangle, \eeqa{5.19}
\beqa E_{1s}=\langle \chi_{1}(\tau_{2}\epsilon_{2}+\tau_{3}\epsilon_{3}) \rangle=2\mu_{1}\langle\chi_{1}(\epsilon_{2}^{2}+\epsilon_{3}^{2})\rangle, \eeqa{5.20}
\beqa E_{2s}=\langle \chi_{2}(\tau_{2}\epsilon_{2}+\tau_{3}\epsilon_{3}) \rangle=2\mu_{2}\langle\chi_{2}(\epsilon_{2}^{2}+\epsilon_{3}^{2})\rangle, \eeqa{5.21}
where $\chi$ is the indicator function for each phase, numbers denote the phase, $b$ denotes bulk component, $s$ denotes the shear component, $\kappa$ is the bulk modulus, and $\mu$ is the shear modulus. These quantities cannot individually be determined from boundary measurements, but Milton and Nguyen \cite{Milton:2011:BVF} correlate them through inequalities. 

From \eq{5.18}-\eq{5.21}, we obtain
\beq E=E_{1b}+E_{2b}+E_{1s}+E_{2s}, \eeq{5.22}
\beq a=\kappa_{1}E_{1b}+\kappa_{2}E_{2b}-\mu_{1}E_{1s}-\mu_{2}E_{2s}, \eeq{5.23}
\beq c \geq \frac{E_{1b}}{4\kappa_{1}}+\frac{E_{2b}}{4\kappa_{2}}-\frac{E_{1s}}{4\mu_{1}}-\frac{E_{2s}}{4\mu_{2}}. \eeq{5.24}

Further inequalities can be obtained using positivity of the variances
\beq \langle(\chi_1 \varepsilon_i - \frac{\chi_1}{f_1}\langle \chi_1\varepsilon_k\rangle)^2\rangle \geq 0,\quad
 \langle(\chi_2 \varepsilon_i - \frac{\chi_2}{f_1}\langle \chi_2\varepsilon_k\rangle)^2\rangle \geq 0,
\eeq{5.24a}
which imply
\beq E_{1b}\geq\frac{A_{1b}}{f_{1}}, \eeq{5.32}
\beq E_{2b}\geq\frac{A_{2b}}{f_{2}}, \eeq{5.33}
\beq E_{1s}\geq\frac{A_{1s}}{f_{1}}, \eeq{5.34}
\beq E_{2s}\geq\frac{A_{2s}}{f_{2}}, \eeq{5.35}
where $f_1$ and $f_2$ are the volume fractions of each phase, and
\beq A_{1b}=2\kappa_{1}\langle\chi_{1}\epsilon_{1}\rangle^{2}, \eeq{5.27}
\beq A_{2b}=2\kappa_{2}\langle\chi_{2}\epsilon_{1}\rangle^{2}, \eeq{5.28}
\beq A_{1s}=2\mu_{1}(\langle\chi_{1}\epsilon_{2}\rangle^2+\langle\chi_{1}\epsilon_{3}\rangle^{2}), \eeq{5.29}
\beq A_{2s}=2\mu_{2}(\langle\chi_{2}\epsilon_{2}\rangle^{2}+\langle\chi_{2}\epsilon_{3}\rangle^{2}). \eeq{5.30}
These four quantities can be determined from the known values of $\langle \BGj \rangle$ and $\langle \nabla \Bu \rangle$
using the relations 
\beqa
	\langle \chi_1\varepsilon_1\rangle & = & \frac{1}{2(\kappa_2 - \kappa_1)}(2\kappa_2\langle\varepsilon_1\rangle - \langle \sigma_1\rangle), \quad \langle \chi_2\varepsilon_1\rangle = \frac{1}{2(\kappa_1 - \kappa_2)}(2\kappa_1\langle\varepsilon_1\rangle - \langle \sigma_1
	\rangle), \nonum
	\langle\chi_1 \varepsilon_j\rangle & = & \frac{1}{2(\mu_2 - \mu_1)}(2\mu_2\langle \varepsilon_j\rangle - \langle \sigma_j\rangle), \quad \langle\chi_1 \varepsilon_j\rangle = \frac{1}{2(\mu_2 - \mu_1)}(2\mu_2\langle \varepsilon_j\rangle - \langle \sigma_j\rangle), \quad j = 2, 3. \nonum &~ &
\eeqa{5.30a}

For this analysis, we focus our view on phase one.
Solving \eq{5.22} and \eq{5.23} for $E_{2b}$ and $E_{2s}$ yields
\beq E_{2b}=\frac{a+E\mu_{2}-E_{1b}(\kappa_{1}+\mu_{2})+E_{1s}(\mu_{1}-\mu_{2})}{\kappa_{2}+\mu_{2}},\eeq{5.25}
and
\beq E_{2s}=\frac{(\Gk_1-\Gk_2)E_{1b}-(\Gm_1+\Gk_2)E_{1s}-a+E\Gk_{2}}{\kappa_{2}+\mu_{2}}.\eeq{5.26}
Plugging \eq{5.25} and \eq{5.26} into \eq{5.33} and \eq{5.35} gives
\beq
\frac{a+E\mu_{2}-E_{1b}(\kappa_{1}+\mu_{2})+E_{1s}(\mu_{1}-\mu_{2})}{\kappa_{2}+\mu_{2}}
\geq\frac{A_{2b}}{f_{2}},
\eeq{5.36}
and
\beq \frac{(\Gk_1-\Gk_2)E_{1b}-(\Gm_1+\Gk_2)E_{1s}-a+E\Gk_{2}}{\kappa_{2}+\mu_{2}}
\geq\frac{A_{2s}}{f_{2}}, \eeq{5.37}
and \eq{5.24} becomes
\beq 4\kappa_{2}\mu_{2}c\geq E(\mu_{2}-\kappa_{2})-\frac{E_{1b}}{\kappa_{1}}(\mu_{2}+\kappa_{1})(\kappa_{1}-\kappa_{2})+\frac{E_{1s}}{\mu_{1}}(\mu_{1}+\kappa_{2})(\mu_{1}-\mu_{2})+a. \eeq{5.38}
So \eq{5.32}, \eq{5.34}, \eq{5.36}, \eq{5.37}, and \eq{5.38} bound a feasible region in the $(E_{1b}, E_{1s})$ plane that, in the case $\Gm_1>\Gm_2$ 
and $\Gk_1>\Gk_2$, might resemble Figure 5.
\begin{figure}[htp]
	\centering
	\includegraphics[width=0.50\textwidth]{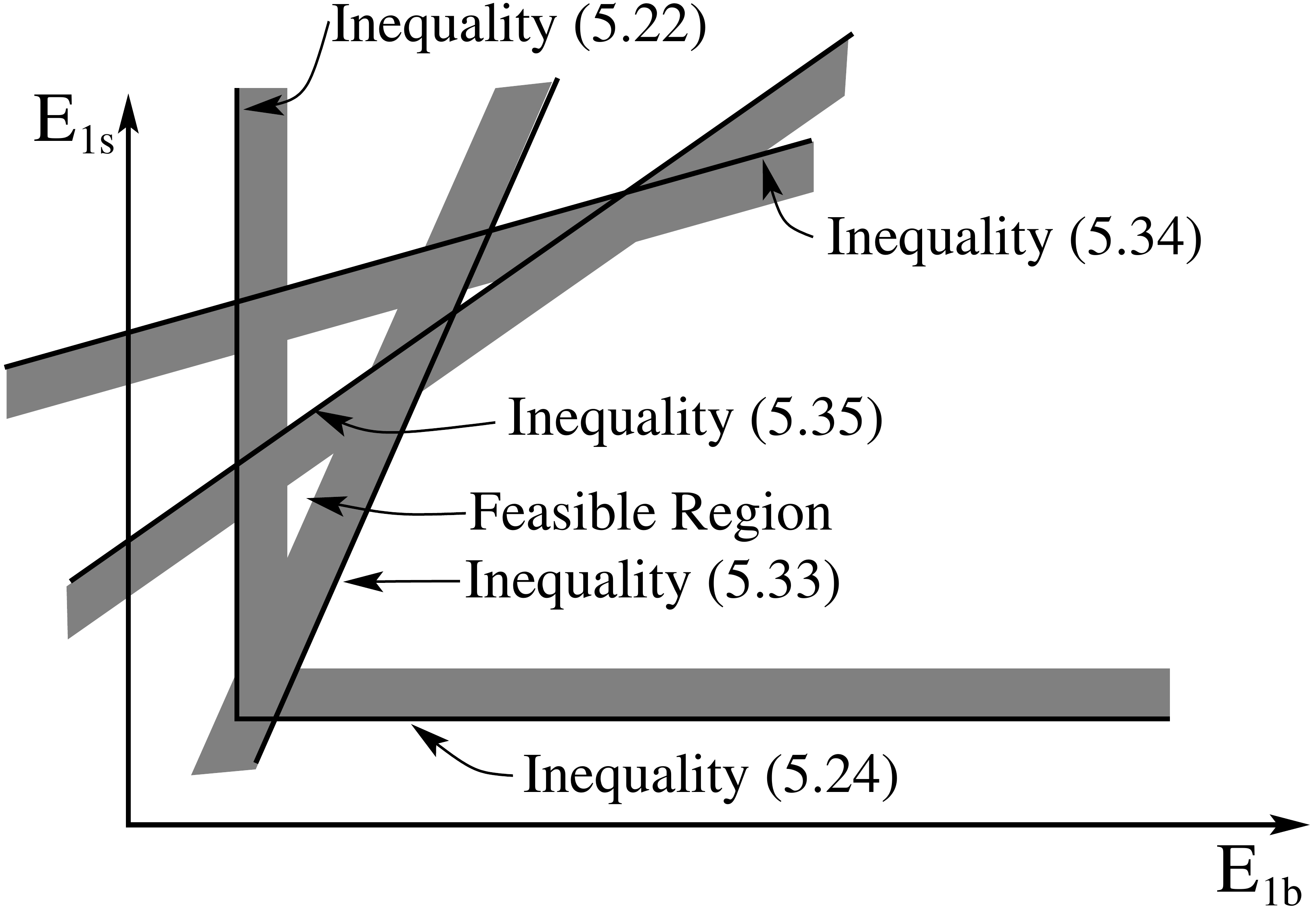}
	\caption{The feasible region}
	\label{fig:elasticity}
\end{figure}

If the material has not yielded the Von Mises-Tresca criterion \eq{5.14a} implies 
\beq \Ge_{2}^{2}+\Ge_{3}^{2}\leq k_1, \eeq{5.38a}
at each point in phase 1, where the threshold $k_1$ depends on the material properties of phase 1. This implies
\beq E_{1s}\leq 2\Gm_1f_1k_1 \eeq{5.38b}
which defines a region in the $(E_{1b}, E_{1s})$ plane that is compatible with the Von Mises-Tresca criterion. Other yield criteria
would yield different compatible regions in the $(E_{1b}, E_{1s})$ plane. If the feasible region does not intersect the compatible region
then the body must have yielded. (One cannot conclude that it is phase 1, rather than phase 2 which has yielded because if either phase
yields, the linear equations of elasticity no longer apply). A similar analysis applies to phase 2, by switching the subscripts 1 and 2.  

If we have additional information about the response of the body to slowly oscillating boundary displacement fields then
the feasible region can be reduced down to a point. If the displacement field at the boundary is $\Bu_0$ without the oscillations,
let it be the real part of $e^{i\Go t}\Bu_0$ with the oscillations, where $t$ is the time and $\Go$ is the frequency, which is small. We can
forget about the factor of $e^{i\Go t}$ since this will factor out of all equations, due to linearity. Thus, mathematically
the boundary displacement can be kept at $\Bu_0$ but the internal fields and the moduli will  become complex due to viscoelasticity.
If the frequency is low, we can use the quasistatic elasticity equations, and the elasticity tensor will be perturbed from $\BC(\Bx)$ to $\BC(\Bx)+\Gd\BC(\Bx)$ while the internal displacement field will be perturbed from $\Bu(\Bx)$ to $\Bu(\Bx)+\Gd\Bu(\Bx)$,  where $\Gd\BC(\Bx)$ and $\Gd\Bu(\Bx)$ are small and complex. Introducing
the strain $\BGe=[\Grad\Bu+(\Grad\Bu)^T]/2$ and its complex perturbation $\Gd\BGe$, we first note that with the perturbation
$$\int_{\Omega}\BGe:\BC\BGe \quad {\rm changes~to} \quad \int_{\Omega}(\BGe+\delta\BGe):(\BC+\delta \BC)(\BGe+ \delta\BGe),$$
and both these quantities can be obtained, using integration by parts, from the boundary values of $\Bu$ and $\BGj\Bn$. To second order in the perturbation we have
\beq \int_{\Omega}(\BGe+\delta\BGe):(\BC+\delta \BC)(\BGe+ \delta\BGe) \approx \int_{\Omega}\BGe:\BC\BGe + 2\int_{\Omega}\delta\BGe:\BC\BGe + \int_{\Omega}\BGe:\delta\BC\BGe, \eeq{5.39}
and
\beqa \int_{\Omega}\delta\BGe:\BC\BGe= \int_{\delta\Omega}\delta \Bu\cdot(\BGj\Bn)dS = 0, \eeqa{5.40} 
since $\delta \Bu = 0$ on $\Md\Omega$. So the quantity
\beqa \Gd E=\frac{1}{\mid\Omega\mid}\int\epsilon\delta C\epsilon & = & 2\delta\kappa_{1}\langle\chi_{1}\epsilon_{1}^{2}\rangle+2\delta\mu_{1}\langle\chi_{2}(\epsilon_{2}^{2}+\epsilon_{3}^{2})\rangle+2\delta\kappa_{2}\langle\chi_{2}\epsilon_{1}^{2}\rangle+2\delta\mu_{2}\langle\chi_{2}(\epsilon_{2}^{2}+\epsilon_{3}^{2})\rangle \nonum
&=& 2\frac{\delta\kappa_{1}}{\kappa_{1}}E_{1b}+2\frac{\delta\mu_{1}}{\mu_{1}}E_{1s}+2\frac{\delta\kappa_{2}}{\kappa_{2}}E_{2b}+2\frac{\delta\mu_{2}}{\mu_{2}}E_{2s} \eeqa{5.41}
can be approximately determined from boundary measurements. If $\delta\kappa_{1}$, $\delta\mu_{1}$, $\delta\kappa_{2}$, $\delta\mu_{2}$ are complex, then
\beq \Real\Gd E=2\Real(\frac{\delta\kappa_{1}}{\kappa_{1}})E_{1b}+2\Real(\frac{\delta\mu_{1}}{\mu_{1}})E_{1s}+2\Real(\frac{\delta\kappa_{2}}{\kappa_{2}})E_{2b}+2\Real(\frac{\delta\mu_{2}}{\mu_{2}})E_{2s}, \eeq{5.42}
and
\beq \Imag\Gd E=2\Imag(\frac{\delta\kappa_{1}}{\kappa_{1}})E_{1b}+2\Imag(\frac{\delta\mu_{1}}{\mu_{1}})E_{1s}+2\Imag(\frac{\delta\kappa_{2}}{\kappa_{2}})E_{2b}+2\Imag(\frac{\delta\mu_{2}}{\mu_{2}})E_{2s}, \eeq{5.43}
are approximately known.

In the generic case, where there is no degeneracy, the four equations \eq{5.22}, \eq{5.23}, \eq{5.42} and \eq{5.43}, can be solved for $E_{1b}, E_{1s}, E_{2b}$ and $E_{2s}$.
If the point $(E_{1b}, E_{1s})$ lies outside the compatible region for phase 1 or if $(E_{2b}, E_{2s})$ lies outside the compatible region for phase 2, then the body must
have yielded. 

\section{$E_\GO$  Inclusions}
\setcounter{equation}{0}

In this section we give a brief overview of the method of Kang, Kim and Milton \cite{Kang:2011:SBV} for finding optimal $E_\GO$  inclusions in two-dimensions. These are
defined as an inclusion of phase $1$ inside the body $\GO$ such that for appropriate boundary conditions the field inside the inclusion is constant. 
These inclusions are of interest to us because for $E_\GO$ inclusions, with the appropriate boundary conditions on the fields, many of the previously discussed bounds are optimal. We assume the inclusion is simply connected
and lying strictly within the simply connected body $\GO$. 
Coordinates are chosen so the $x$-axis is aligned with the field inside the inclusion, and so the projection of the inclusion onto the $y$-axis is the interval $[-1,1]$.
Then the constancy of the field 
is formulated as $V(x,y) = x$ in $E_\GO$.  The potential satisfies the standard conductivity equations
\beq
 \BJ(x,y)=\Gs(x,y) \BE(x,y), \quad \BE(x,y)=-\Grad V(x,y),\quad \Curl [\BR_\perp \BJ(x,y)]=0,
\eeq{6.1}
where $\BR_\perp$ is the rotation matrix \eq{2.18} for a $90^\circ$ degree rotation.  Considering the divergence of $\BJ$ as the curl of the rotated $\BJ$ field allows us to define a potental W such that:
  	 \beq
   	  \BR_\perp \BJ(x,y)=\Grad W.
  	 \eeq{6.5}
So in phase 2, which without loss of generality we assume to have conductivity $\Gs_2=1$, we have that $ \BR_\perp \Grad W=\Grad V$. Equivalently $W$ and $V$ satisfy the Cauchy-Riemann equations and thus $V+iW$ is an analytic function of $x+iy$ in phase 2.  In phase 1, $\BR_\perp \Grad W=\Gs_1\Grad V=\Gs_1\Grad x$, and since the potentials
$V$ and $W$ are continuous across the boundary we have that
\beq
 V=x, \quad  W = \Gs_1 y, \quad {\rm on} ~ 
\Md E_\GO.
\eeq{6.6}
Next define the potential: 
\beq
 v+iw=\frac{i(V+iW-z)}{1-\Gs_1},
\eeq{6.7}
which is an analytic function of $z=x+iy$
in $\GO\setminus E_\GO$ and we see that on $\Md E_\GO$
\beq
 v=y, \quad w=0.
\eeq{6.8}
As is often useful for solving two dimensional free boundary problems involving the Laplace equation, Kang, Kim, and Milton \cite{Kang:2011:SBV}  use a hodograph transform.  To do this,  assume that $v+iw$ is a univalent function of $x+iy$ outside of $E_\GO$ and thus $z=x+iy$ is an analytic function of $h=v+iw$. Then the image of $E_\GO$ 
is a slit on the $v$ axis (where $w=0$) from $v=-1$ to $v=1$.  We want to find functions $x+iy$ of $h=v+iw$ such that $y=v$ on the slit. It is helpful to consider the function $\bar z=\bar x+i\bar y=z-ih=x+iy-iv+w$ which on the slit has $\bar y=y-v=0$.  Now make the fractional linear transformation 
\beq s=\frac{1-h}{1+h}=\frac{1-[v+iw]}{v+iw+1},
\eeq{6.10} which maps $ h=1$ to $s=0$ and $h=-1$ to $s=\infty$.  So in the $s$-plane the slit becomes the positive real axis. Next, the square root transformation is used to map the positive real axis to the entire real axis, namely $t=\sqrt s$ where $\sqrt s$ is chosen with a branch cut on the positive real axis.  Thus $\bar y=0$ on the entire real $t$
axis. This is satisfied by taking $\bar z=f(t)$ where $f(t)$ satisfies $f(t^*)=(f(t))^*$ and $a^*$ is the complex conjugate of $a$.  To satisfy this we could set
\beq f(t)=\sum_{\Ga=1}^n [\frac{ b_\Ga}{t-t_\Ga}+\frac{b_\Ga^*}{t-t_\Ga^*}]+c,
\eeq{6.11}
where the $t_\Ga$ are complex with nonzero imaginary components (to ensure $f(t)$ has no poles on the real axis), the $b_\Ga$ are real or complex, and $c$ is real. 
Tracing back the formulae,
we see that
\beq z=z(h)=ih+f\left(\sqrt{\frac{1-h}{1+h}}\right), \eeq{6.11a}
and since $h=y$ on the slit, the boundary of the $E_\GO$ inclusion is given by the formula
\beq x=f\left(\pm\sqrt{\frac{1-y}{1+y}}\right). \eeq{6.11b}
To avoid self intersections it is required that $f(t) \neq f(-t)$ for all real $t\ne 0$. Additionally, to ensure the univalence of $z(h)$ in the neighbourhood of the slit ends
 $h=-1$ and $h=1$ it is required that the derivative $f'(0)$ is nonzero and $f(t)$ has the asymptotic expansion
\beq f(t)=\Gb_0+\Gb_1/t+\mathcal{O}(|t|^{-2})\quad {\rm as}~|t|\to\infty, \eeq{6.11c}
where $\Gb_1$ is real and positive.

Kang, Kim and Milton \cite{Kang:2011:SBV} gave some numerical examples illustrating $E_\GO$ inclusions. Figure 6 shows a further example of an $E_\GO$ inclusion
and the function which generates it.  By taking functions $f(t)$ of the form \eq{6.11} with $n\leq 5$ and real or complex residues and their conjugates 
one can generate a wide variety of $E_\GO$ inclusion shapes, as shown in Figure 7. 

\begin{figure}
	\centering
	\includegraphics[width=0.5\textwidth]{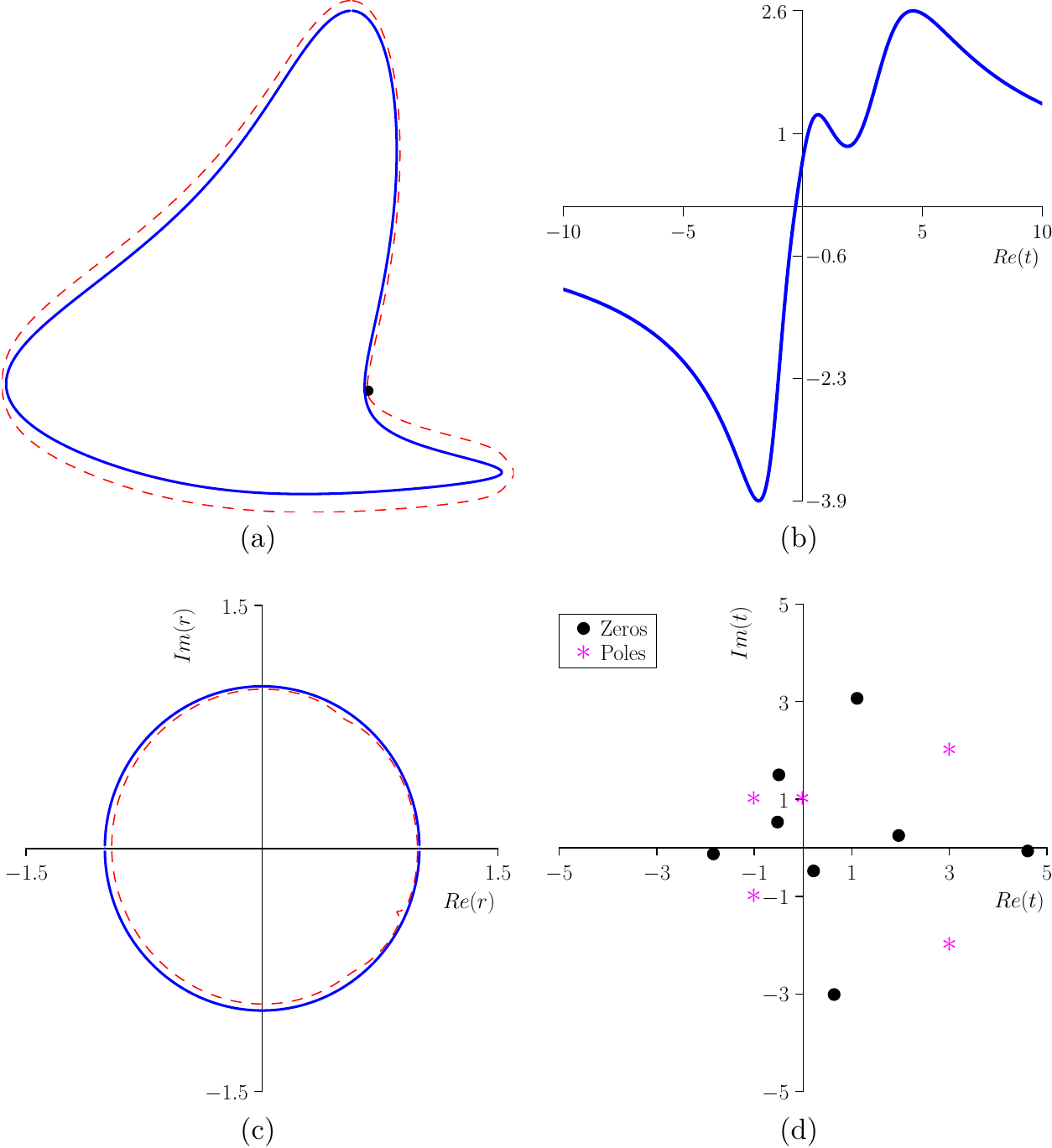}
	\caption{An example of an $E_\GO$ inclusion is shown in (a) given by the blue outline, with a possible boundary of $\GO$ marked by the dashed red line.
Shown in (b) is the function $f(t)$ which generates this inclusion. Shown in (c) are blue and dashed red curves in the $r=(t-i)/(t+i)$ plane the images of which
under the mapping $z(h(t(r)))$ with $h(t)=(1-t^2)/(1+t^2)$ and $t(r)=i(1+r)/(1-r)$ give the blue and dashed red curves in figure (a). Shown in (d) are the poles
and zeros of the function $dz(h(t))/dt$. At these zeros in the upper half $t$-plane the map $z(h(t))$ is not conformal and as a consequence these zeros 
map to points in the $z$-plane where $v+iw$ is not a univalent function of $x+iy$. Such a point is indicated by the small black circle in (a): the boundary
of $\GO$ must pass between it and the boundary of the $E_\GO$ inclusion. }
	\label{fig:Patrick2}
\end{figure}
\begin{figure}
	\centering
	\includegraphics[width=0.5\textwidth]{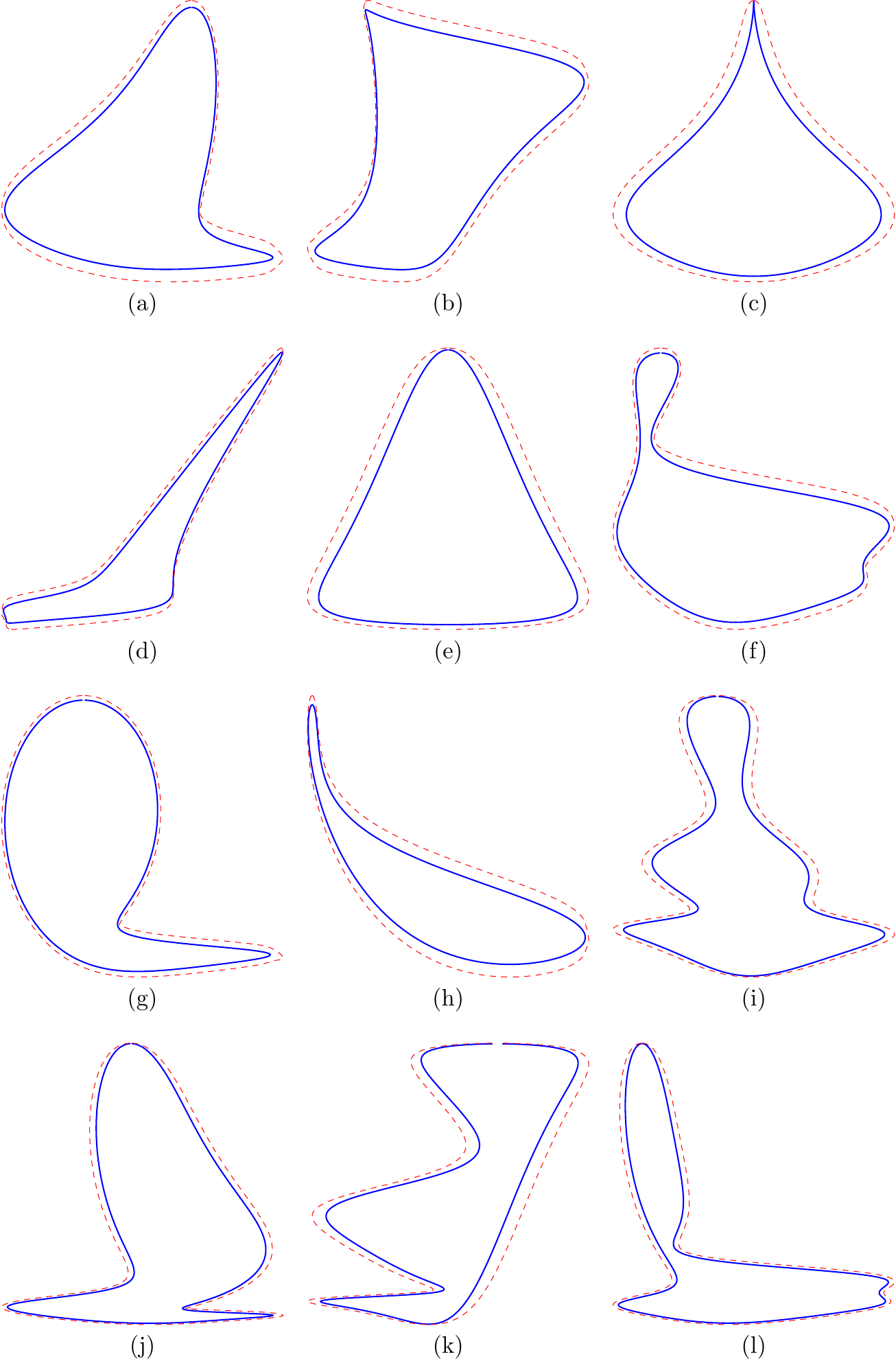}
	\caption{Examples showing the wide variety of $E_\GO$ inclusion shapes that can be generated from the formula \eq{6.11} with $n\leq 5$. Each  $E_\GO$ 
inclusion is outlined in blue, and the surrounding red dashed line represents one of many possible boundaries for $\GO$.}
	\label{fig:Patrick1}
\end{figure}

\subsection{Shearing and stretching of $E_\GO$ inclusions}
In two-dimensions periodic microstructures having the property that the field is uniform in phase can be transformed to other geometries having the same property:
see section 23.9 of \cite{Milton:2002:TOC}. In this transformation the shape of each inclusion undergoes an affine transformation,  which is different to the affine transformation that the unit cell of periodicity undergoes. Here we apply a similar analysis to 
show that our $E_\GO$ inclusions remain $E_\GO$ inclusions after appropriate shears and stretches. 
Again suppose that $z=x+iy$ is an analytic function of $h=v+iw$ in the neighborhood of the slit $w=0$, $-1\leq v \leq 1$, and $y=v$ on the slit. Now observe that $z'=x'+iy'=\Gg_1(x+iy)+\Gg_2(v+iw)$  is still an analytic function of $v+iw$ for all choices of $\Gg_1$ and $\Gg_2$. If we choose $\Gg_1$ to be real then $y'$ will not depend on $x$ and will be
proportional to $y$ along the slit.  Along the slit $w=0$, and we have   
\beq
 y'=\Gg_1y+\Imag(\Gg_2)v=(\Gg_1+\Imag(\Gg_2))y ,\quad x'=\Gg_1x+\Real(\Gg_2)v=\Gg_1x+\Real(\Gg_2)y.
\eeq{6.12}
So if we choose $\Imag(\Gg_2)=1-\Gg_1$ we ensure that $y'=y=v$ along the slit. In other words the function $z'(h)$ satisfies the same desired properties as $z(h)$, and 
associated with it there is an inclusion having a constant field inside, with boundary
\beq
x'=\Gg_1x+\Real(\Gg_2)y,  \quad y'=y\quad {\rm where} ~ (x,y)\in \Md E_\GO.
\eeq{6.13}
When $\Gg_1=1$ and $\Real(\Gg_2)\ne 0$ this corresponds to a shear of the inclusion, and when $\Real(\Gg_2)=0$ and $\Gg_1\ne 1$ it corresponds to a stretch of the inclusion in the $x$-direction: more generally it is a
combination of the two transformations. 

\subsection{An additional field supported by the $E_\GO$ inclusion}
As shown in \cite{Kang:2011:SBV} the $E_\GO$ inclusion can also support a field which is constant and aligned with the $y$-axis. To review this, let $V'$ and $W'$ be the associated potentials, with
fields $\BE(x,y)=-\Grad V'$  and $\BJ(x,y)=-\BR_\perp \Grad W'$. If inside the inclusion $V'=y$, then $\BR_\perp \Grad W'=\Gs_1\Grad V'=\Gs_1\Grad y$ implying
$W'=-\Gs_1x$. Since these potentials are continuous across the inclusion boundary we have that $V'=y, W'=-\Gs_1y$ on $\partial E_\GO$. Outside the inclusion
(assuming $\Gs_2=1$) $V'+iW'$ must be an analytic function of $z=x+iy$. 
We look for a solution with 
\beq
V'+iW'=\Ga(V+iW)+\Gb z,
\eeq{6.15} 
where $\Ga$ and $\Gb$ are complex constants. Using the boundary values of $V$ and $W$ we have
\beq
V'+iW'=\Ga(x+i\Gs_1y)+\Gb(x+iy) \quad {\rm on} ~ \Md E_\GO.
\eeq{6.15a}
The complex constants $\Ga$ and $\Gb$ are chosen so  $V'$ and $W'$ satisfy the boundary conditions which gives
\beqa
V'& = & \Real(\Ga)x-\Imag(\Ga)\Gs_1y+\Real(\Gb)x-\Imag(\Gb)y=y, \nonum
W'& = &\Imag(\Ga)x+\Real(\Ga)\Gs_1y+\Imag(\Gb)x+\Real(\Gb)y=-\Gs_1x.
\eeqa{6.16}
These are satisfied if $\Ga$ and $\Gb$ take the purely imaginary values 
\beq \Ga=i,\quad \Gb=-i(1+\Gs_1).
\eeq{6.17}
Thus the inclusion can also support a constant field in this orthogonal direction, and by superposition in any direction.  
\subsection{Elastic $E_\GO$  Inclusions}
It was recognized that isotropic composites of two isotropic phases which achieve the Hashin-Shtrikman bounds on the effective conductivity
also necessarily achieve the Hashin-Shtrikman bounds on the effective bulk modulus \cite{Berryman:1988:MRC, Milton:1984:CEE}
(see also \cite{Gibiansky:1996:CBC} and references therein). The condition that these
bounds be achieved is that the field is uniform in one phase: thus uniformity of the electric and current fields in a phase, implies uniformity 
of the stress and strain fields within that phase, and vice-versa. A deeper reason for this connection was found by \cite{Grabovsky:1996:BEM}, who
discovered that in these geometries, fields solving the conductivity equations can be mapped to fields solving the elasticity equations and vice-versa. 
One would expect a similar mapping to hold for $E_\GO$ inclusions and we will now directly see this is the case. 

We will now use the potentials $V$, $W$, $V'$ and $W'$ to construct stress and strain fields which solve the elasticity equations, with the fields being uniform and hydrostatic in the $E_\GO$ inclusion and with the materials being isotropic in both phases.
Consider
\beq
\BGve= \begin{bmatrix}
{\partial V}/{\partial x} & {\partial V'}/{\partial x} \\
{\partial V}/{\partial y} & {\partial V'}/{\partial y}
\end{bmatrix},
\eeq{6.18}
which we will interpret as a strain field, associated with the displacement $\Bu=(V,V')$. In the $E_\GO$ inclusion,
since $V=x$ and $V'=y$, we have $\BGve=\BI$. 
Let us establish that outside the inclusion $\BGve$ is symmetric and $\Tr(\BGve)$ is constant.  From \eq{6.15a} and \eq{6.17} we see that  
\beq
V'=-W+(1+\Gs_1)y, \quad W'=V-(1+\Gs_1)x,
\eeq{6.18a}
which implies
\beqa
\frac{\Md V'}{\Md x}& =& -\frac{\Md W}{\Md x}= \frac{\Md V}{\Md y}, \nonum
\frac{\Md V'}{\Md y}& = & -\frac{\Md W}{\Md y}+1+\Gs_1=-\frac{\Md V}{\Md x}+1+\Gs_1,
\eeqa{6.18b}
where we have used the fact that $V$ and $W$ satisfy the Cauchy Riemann equations.
Thus $\BGve$ is symmetric, which implies $\BGve$ is the symmetrized gradient of the displacement $\Bu=(V,V')$, and $\Tr(\BGve)$ is constant.
To construct solutions to the elasticity problem we want to choose $\Gs_1$ such that $\BGve$ satisfies the elasticity equations,  
\beq
\BGj=\Gl(x,y)(\Tr(\BGve))I+2\Gm(x,y)\BGve, \quad \Div \BGj = 0.
\eeq{6.21}
Then the stress in the inclusion
\beq \BGj=2(\Gl_1+\Gm_1)\BI \eeq{6.21a}
is clearly divergence free, and the stress in the matrix
\beq \BGj=\Gl_2(1+\Gs_1)\BI+2\Gm_2\BGve \eeq{6.21b}
is also divergence free because $V$ and $V'$ are harmonic functions. 
We also require the tractions to be continuous across $\Md E_\GO$. As $V$ and $V'$ both solve the conductivity equations it follows that
$$ \Gs \Bn^T\begin{bmatrix} {\Md V}/{\Md x}\cr {\Md V}/{\Md y}\end{bmatrix},\quad
\Gs \Bn^T\begin{bmatrix} {\Md V'}/{\Md x}\cr {\Md V}/{\Md y}\end{bmatrix} $$
are both continuous across $\Md E_\GO$, where $\Bn$ is the outward normal to the interface. Recalling that $\Gs_2=1$, this implies 
\beqa
\Bn^T\begin{bmatrix} {\Md V}/{\Md x}\\ {\Md V}/{\Md y}\end{bmatrix} = \Gs_1 \Bn^T\begin{bmatrix} 1\\ 0 \end{bmatrix}, \\
\Bn^T\begin{bmatrix} {\Md V'}/{\Md x}\\ {\Md V'}/{\Md y}\end{bmatrix} = \Gs_1 \Bn^T\begin{bmatrix} 0\\ 1 \end{bmatrix},
\eeqa {6.22}
where on the left $V$ and $V'$ are the potentials just outside the $E_\GO$ inclusion. Hence we deduce that
\beq \Bn^T\BGve=\Gs_1\Bn^T \eeq{6.23}
where on the left $\BGve$ is the field  just outside the $E_\GO$ inclusion.
On the other hand, from \eq{6.21a} and \eq{6.21b}, the continuity of $\Bn^T\BGj$ (which is equivalent to continuity of the
traction $\BGj\Bn$) requires that  
\beq
2(\Gl_1+\Gm_1)\Bn^T=\Gl_2(1+\Gs_1)\Bn^T+2\Gm_2\Bn^T\BGve=[\Gl_2+\Gs_1(\Gl_2+2\Gm_2)]\Bn^T,
\eeq{6.23a}
where we have used \eq{6.23} to eliminate $\BGve$, the field  just outside the $E_\GO$ inclusion.
So the traction is continuous, and the elasticity equations are satisfied, if we choose
\beq
\Gs_1=\frac{2(\Gl_1+\Gm_1)-\Gl_2}{\Gl_2+2\Gm_2}.
\eeq{6.26}
Thus there is a close connection between the $E_\GO$ inclusions for the conductivity and elasticity cases. We have not explored the question as to the shapes of
elastic $E_\GO$ inclusion for which the stress field inside the inclusion is constant, but not hydrostatic. Are such inclusions necessarily elliptical in shape
(or ellipsoidal in three dimensions)? 

\section*{Acknowledgements}
G.W. Milton thanks the National Science Foundation for support through grant DMS-1211359, and all  authors thank the University of Utah for helping support this research, through the Math 4800 undergraduate class. Andrew Boyles is thanked
for his participation in the work and for helping with the section on two-dimensional elasticity.

\end{document}